\documentclass[final,onefignum,onetabnum]{siamart171218}

\usepackage{inputenc}
\usepackage{amsmath,amssymb,amsfonts,color,bm}
\usepackage{graphicx}
\usepackage{epstopdf}
\usepackage{latexsym}
\usepackage{multirow}
\usepackage{float}
\usepackage{verbatim}
\usepackage{enumitem}
\usepackage{algorithmic}
\usepackage{amsopn}
\usepackage{enumitem}

\def\R{{\mathbb R}}

\def\N{{\mathbb N}}

\newcommand{\ds}{\displaystyle}

\newcommand{\de}{\mathrm{d}}

\ifpdf
  \DeclareGraphicsExtensions{.eps,.pdf,.png,.jpg}
\else
  \DeclareGraphicsExtensions{.eps}
\fi

\newsiamremark{remark}{Remark}
\newsiamremark{example}{Example}
\newsiamremark{property}{Property}
\newsiamthm{assumption}{Assumption}
\newsiamremark{hypothesis}{Hypothesis}
\crefname{hypothesis}{Hypothesis}{Hypotheses}
\newsiamthm{claim}{Claim}

\headers{Polynomial Interpolation of Function Averages on Interval Segments}{L. Bruni Bruno, W. Erb}

\title{Polynomial Interpolation of Function Averages on Interval Segments}

\author{L. Bruni Bruno\thanks{Dipartimento di Matematica \lq\lq Tullio Levi-Civita\rq\rq, Universit\`a di Padova, \email{bruni@math.unipd.it}}
\and W. Erb\thanks{Dipartimento di Matematica \lq\lq Tullio Levi-Civita\rq\rq, Universit\`a di Padova, \email{erb@math.unipd.it}}}

\ifpdf
\hypersetup{
  pdftitle={Polynomial Interpolation of Function Averages on Interval Segments},
  pdfauthor={L. Bruni Bruno, W. Erb}
}
\fi

\usepackage[twoside=false]{geometry}

\begin{document}

\maketitle

\begin{abstract}
Motivated by polynomial approximations of differential forms, we study analytical and numerical properties of a polynomial interpolation problem that relies on function averages over interval segments. The usage of segment data gives rise to new theoretical and practical aspects that distinguish this problem considerably from classical nodal interpolation. We will analyse fundamental mathematical properties of this problem as existence, uniqueness and numerical conditioning of its solution. We will provide concrete conditions for unisolvence, explicit Lagrange-type basis systems for its representation, and a numerical method for its solution. To study the numerical conditioning, we will provide concrete bounds of the Lebesgue constant in a few distinguished cases.      
\end{abstract}



\section{Introduction} \label{sec1}

Data interpolation is a pervasive tool in numerical analysis and a cornerstone of modern algorithms in data science and machine learning. At its simplest, this technique consists in constructing a univariate polynomial from prescribed function values. In more abstract frameworks, interpolation is described as determining the element of a linear subspace based on the information provided by a set of linear functionals, see \cite[Section 3.2]{Atkinson1998} or \cite[Chapter II]{Davis75}.  

For Whitney differential forms \cite{BossavitBook, Whitneybook}, one may formulate such a generalized interpolation problem in the following way \cite{Rapetti07,RB09}: is it possible to stably reconstruct, in the sense of \cite{GenLeb}, a polynomial approximant of a Whitney form by its integrals over simplicial subdomains? When the dimension of the ambient space is greater than one, the answer is mostly conjectural, and only numerical solutions have been offered to this problem so far \cite{BruniThesis}. On the other hand, for differential forms in one dimension, this problem can be boiled down to the reconstruction of a polynomial from function averages on interval segments (see Fig. \ref{fig:firstexamples}). This allows us to utilize well-known tools from univariate interpolation and approximation theory \cite{Achieser92, Cheney66, trefappr, Trefethen} in order to obtain concrete theoretical results about unisolvence and well-posedness of this problem, and this is exactly what we will do in this work.  

\begin{figure}[htbp]
    \centering
    \includegraphics[width=7.5cm]{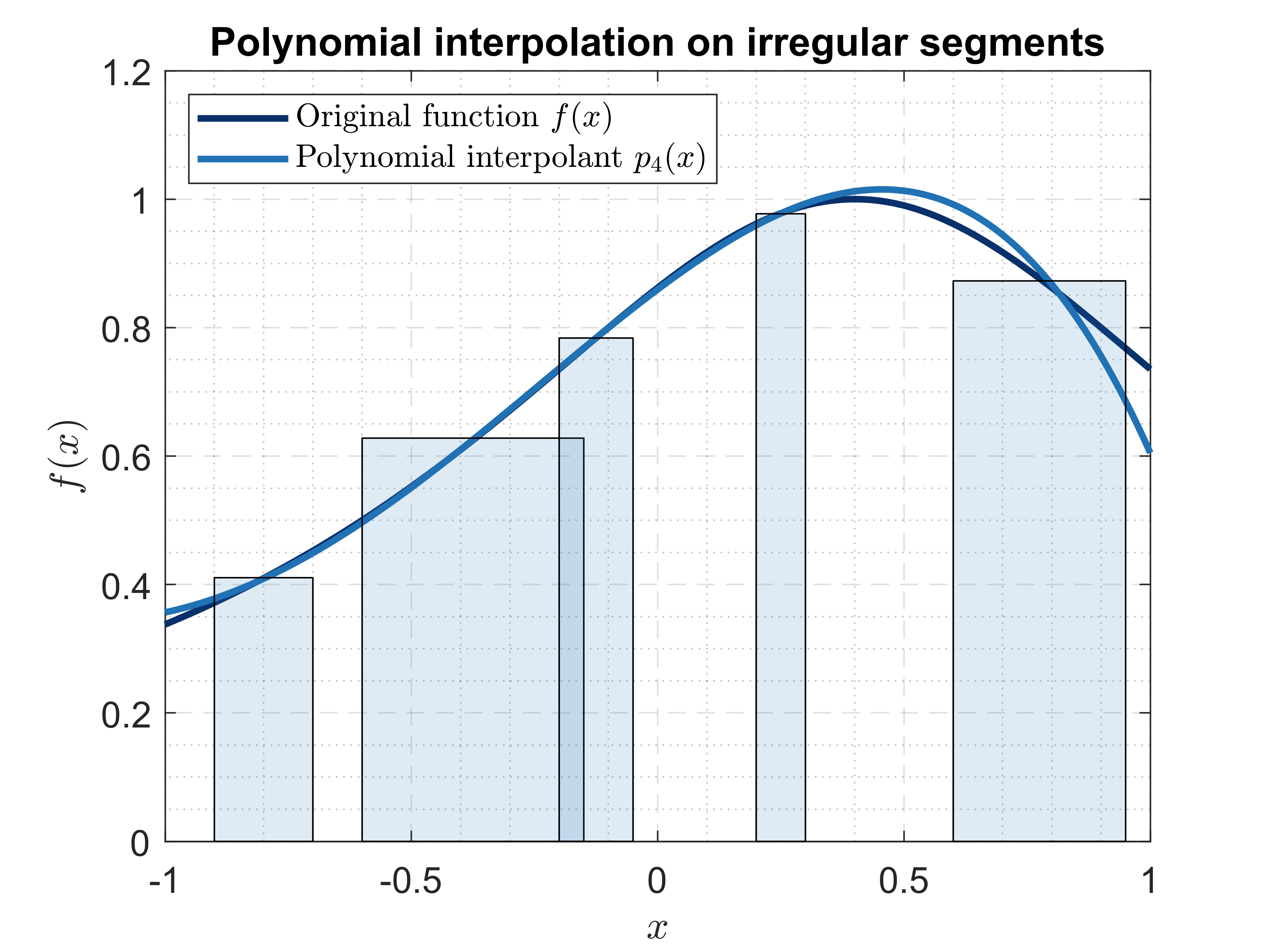}
    \includegraphics[width=7.5cm]{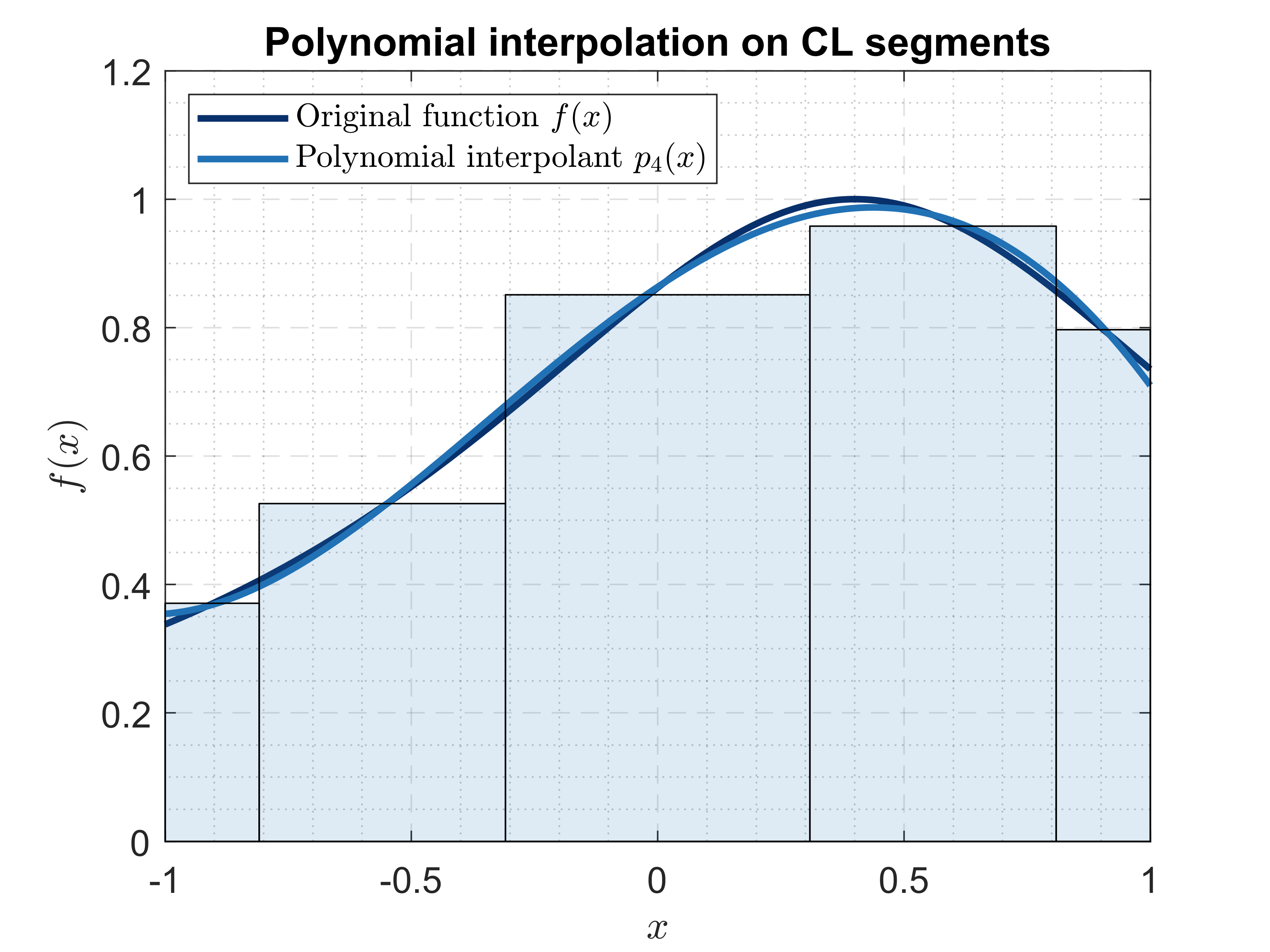}
    \vspace{-.7cm}
    \caption{Examples of polynomial interpolation associated to function averages on $5$ segments. The area of the rectangles corresponds to the integral of the function $f$ (or the interpolating polynomial) over the respective segments. Left: irregular partially overlapping segments. Right: Chebyshev-Lobatto segments in the classes \ref{itm:one} and \ref{itm:two}.}
    \label{fig:firstexamples}
\end{figure}

\subsection*{Problem formulation} \label{sec-problem}
We consider an essentially bounded function $f: I \to \R$ on the interval $I = [-1,1]$. The given data $\mu_i$, $i \in \{1, \ldots, r\}$ consists of integrals of $f$ on $r$ interval segments $s_i = [\alpha_i, \beta_i] \subseteq I$ with $-1 \leq \alpha_i < \beta_i \leq 1$ for $i \in \{1, \ldots, r\}$. More concretely, we have the following integral measurements of $f$ at disposition: 
\begin{equation} \label{eq:segmentdata}
\mu_i = \mu(f,s_i) \doteq \int_{s_i} f(x) \de x = \int_{\alpha_i}^{\beta_i} f(x) \de x, \quad i \in \{1, \ldots, r\}.
\end{equation} 
From the $r$ measurements \eqref{eq:segmentdata}, the \emph{interpolation problem} consists in determining a polynomial $p_{r-1} \in \mathbb{P}_{r-1}$ of degree $r-1$ such that the following $r$ interpolation conditions are satisfied:
\begin{equation} \label{eq:interpolationcondition}
\int_{s_i} p_{r-1}(x) \de x = \mu_i, \quad i \in \{1, \ldots, r\}.
\end{equation}
In his core, this problem is similar to nodal polynomial interpolation in which the interpolating polynomial is determined by $r$ function evaluations $f(\xi_i)$ on $r$ node points $-1 \leq \xi_1 < \ldots \xi_{r} \leq 1$. However, the usage of segments instead of nodes causes substantial deviations that have to be taken into account in the numerical analysis. First of all, as the data is given in terms of integrals \eqref{eq:segmentdata} the problem can be stated more generally for the space $L_{\infty}(I)$ of essentially bounded functions, whereas for nodal function evaluations typically the space $C(I)$ of continuous functions has to be considered. Further, to describe a segment $s_i$ we require two parameters (the endpoints $\alpha_i$ and $\beta_i$) in comparison to one parameter $\xi_i$ for the nodes. This makes it more difficult to characterize the cases when the interpolation conditions \eqref{eq:interpolationcondition} provide a unique polynomial $p_{r-1}$. In fact, existence and uniqueness of a solution $p_{r-1}$ cannot always be guaranteed in the segmental setting \eqref{eq:interpolationcondition}. 

\subsection*{Considered scenarios} To obtain concrete results, we will restrict ourselves to the following three particular classes of segments in which the number of free parameters is reduced:
\begin{enumerate}[label = (C\arabic*)]
\item \label{itm:one} \emph{Chains of intervals}: for $r + 1$ nodes $-1 = \xi_0 < \xi_1 < \cdots < \xi_r = 1$ we assume that the segments $s_i$ are given as 
\[s_i = [\xi_{i-1}, \xi_{i}], \quad i \in \{1, \ldots, r\}.\]
In particular, we have $\alpha_i = \xi_{i-1}$ and $\beta_i = \xi$. Furthermore $\cup_{i =1}^r s_i = I$. 
\item \label{itm:two} \emph{Segments with uniform arc-length}: for $r$ values $0 < \tau_1 < \cdots < \tau_r < \pi$ and $0 < \rho< \pi$, we consider the interval segments $s_i$ determined by $\alpha_i = \cos(\tau_i + \rho)$ and $\beta_i = \cos(\tau_i - \rho)$. The length of these segments is given by
\[ |s_i| = \cos(\tau_i - \rho) - \cos(\tau_i + \rho) = 2 \sin(\tau_i) \sin(\rho).\]
If the interval segments $s_i$ are mapped on the unit circle (using the lower half-circle if $\tau_i - \rho$ or $\tau_i + \rho$ are not in $[0,\pi]$) then the arc-length of every mapped segment on the unit circle is equal to $2 \rho$. This distance will be referred to as arc-length of the segments $s_i$. The points $\cos \tau_i$ will be denoted as arc-midpoints of the interval $s_i$.
\item \label{itm:three} \emph{Segments with identical left endpoints}: for $-1\leq \alpha<\beta_1 < \cdots < \beta_r \leq 1$, we consider the $r$ segments $s_i = [\alpha,\beta_i]$, $i \in \{1, \ldots, r\}$. In particular, all left endpoints $\alpha_i = \alpha$ are identical. In this case, we have $s_i \subseteq s_{i+1}$ and $|s_i|< |s_{i+1}|$ for all $i \in \{1, \ldots, r-1\}$.   
\end{enumerate}  

\begin{example}
We give a few examples of segments that will be used throughout the article.
\begin{enumerate}[label = (\roman*)]
\item {\it Equidistant segments:} based on uniform nodes $\mathcal{X}^{\mathrm{eq}} = \{\xi_i^{\mathrm{eq}} = -1 + \frac{2i}{r}\; : \; i \in \{0, \ldots, r\}\}$ the set $\mathcal{S}^{\mathrm{eq}}$ consists of the uniform segments $s_i^{\mathrm{eq}} = [\xi_{i-1}^{\mathrm{eq}},\xi_i^{\mathrm{eq}}]$, $i \in \{1, \ldots, r\}$ having two consecutive nodes of $\mathcal{X}^{\mathrm{eq}}$ as endpoints. The uniform segments $\mathcal{S}^{\mathrm{eq}}$ are in the class \ref{itm:one}. 
\item {\it Chebyshev-Lobatto (CL) segments:} to define the Cheybshev-Lobatto segments in $I$, we use the equidistant nodes $\nu_i = \frac{\pi}{r} i$, $i \in \{0, \ldots, r\}$, on $[0,\pi]$ and introduce the CL nodes as
$\xi_i^{\mathrm{CL}} = \cos(\nu_{r-i})$. Then, the CL segments $s_i^{\mathrm{CL}}$ are given as
\[ s_i^{\mathrm{CL}} = [\xi_{i-1}^{\mathrm{CL}}, \xi_i^{\mathrm{CL}}], \quad i \in \{1, \ldots, r\}. \] The CL segments $\mathcal{S}^{\mathrm{CL}}$ are contained in the classes \ref{itm:one} and \ref{itm:two}. The segments $s_i^{\mathrm{CL}}$ have arc-midpoint $\tau_i = (\nu_{r-i} + \nu_{r-i+1})/2 = \frac{2 (r-i)-1}{2r} \pi$ and uniform arc-radius $\rho = \frac{\pi}{2r}$. The length of the CL segments is given as
\[|s_i^{\mathrm{CL}}| = |\cos(\nu_{r-i}) - \cos(\nu_{r-i+1})| = 2\sin\left(\frac{\pi}{2 r}\right) \sin\left(\frac{2 i-1}{2r} \pi \right).\]
\item {\it Overlapping segments:} the interpolation problem on the two families (i) and (ii) above can be rewritten using overlapping segments. We will do this exemplarily for the CL nodes. We define the overlapping Chebyshev-Lobatto segments $s_i^{\mathrm{CLO}}$ in $\mathcal{S}^{\mathrm{CLO}}$ as
\[s_i^{\mathrm{CLO}} = [-1, \xi_i^{\mathrm{CL}}], \quad i \in \{1, \ldots, r\}.\]
The respective sets $\mathcal{S}^{\mathrm{CLO}}$ and $\mathcal{S}^{\mathrm{eqO}}$ are all contained in the class \ref{itm:three}. Note that, although the segments $\mathcal{S}^{\mathrm{CL}}$ and $\mathcal{S}^{\mathrm{CLO}}$ generate the same polynomial interpolant $p_{r-1}$ for a given function $f$, the class \ref{itm:three} requires different analysis tools than the class \ref{itm:one}. This will get clear in the study of the Lebesgue constants.   
\end{enumerate}
\end{example}

\subsection*{Well-posedness} When solving the interpolation problem \eqref{eq:interpolationcondition} not only existence and uniqueness are relevant, but also the numerical conditioning of the problem that determines how much errors in the input data affect the solution. For nodal interpolation, it is well-known that the wrong placement of the interpolation nodes leads to examples of tremendous ill-posedness \cite{Runge}. The respective numerical conditioning is typically measured in terms of the so-called Lebesgue constant \cite{Brutman1,trefappr}. Consistently with the nodal case, very different behaviours of the Lebesgue constants have been numerically observed also in the segmental case \cite{BruniThesis}. As theoretical estimates are still lacking, our goal is to provide concrete bounds for the segmental Lebesgue constants and to reveal respective relationships to the nodal case. As for nodal interpolation, we will use the uniform norm 
\[ \|f\| = \underset{x \in I}{\mathrm{ess\, sup}} |f(x)|\] 
to measure errors in the input and the output. Since we are dealing with function averages as given information, it is useful to consider the following equivalent formulation of the uniform norm:
\begin{equation} \label{eq:charuniformnorm}
\Vert f \Vert = \sup_{-1\leq a < b \leq 1}  \frac{1}{b-a}  \left  \vert \int_{a}^b f(x) \de x \right\vert.
\end{equation}

\subsection*{Main contributions} 
In this work, we provide a thorough theoretical and numerical analysis of the segmental interpolation problem \eqref{eq:interpolationcondition}. This includes the study of existence and uniqueness properties, an algorithm for the calculation of the interpolating polynomial, as well as the estimate of the segmental Lebesgue constant providing a measure for the numerical conditioning of the problem. In Proposition \ref{prop:lagrbasis} and Proposition \ref{prop:lagrbasisLB}, we compute explicit Lagrange-type bases for the interpolation problem that are dual to function averages over segments. In Proposition \ref{prop:uniquenessC2} we study a class of segments in which non-unisolvence can be characterized explicitly. Regarding numerical conditioning, we show in Theorem \ref{thm:operatornorm} that for non-overlapping segments the Lebesgue constant coincides with the operator norm of the interpolation operator. These results finally converge into Eq. \eqref{eq:uniformbounds}, where we prove that the Lebesgue constant associated with equidistant segments suffers from an exponential growth, and Corollary \ref{cor:logarithmicgrowth}, where we show that Chebyshev-Lobatto segments offer a slow logarithmic growth of the Lebesgue constant.

\subsection*{Outline} In Section \ref{sect:bases}, we start to investigate the segmental interpolation problem by representing it in relevant basis systems. Based on these representations, we derive then in Section \ref{sect:uniqueness} concrete results about the unisolvence of segment sets and offer explicit formulas for the Lagrange bases in the classes \ref{itm:one} and \ref{itm:three}. This is preparatory to Section \ref{sect:Leb}, where the segmental Lebesgue constant is recalled and characterised; moreover, theoretical features of this quantity are proved. Exploiting these results in Section \ref{sect:EstLeb}, we offer theoretical bounds for the aforementioned relevant examples of segments. We dedicate Section \ref{sect:WFperspective} to the parallelism with Whitney forms and gather final conclusions in Section \ref{sec-conclusion}.

\section{Polynomial bases and numerical solution schemes} \label{sect:bases}
To calculate the interpolating polynomial $p_{r-1}$ from the conditions \eqref{eq:interpolationcondition}, a polynomial basis is required. In the following we consider three basis systems that are useful for this purpose.

\subsection{Monomial basis}

We first formulate the interpolation problem in terms of the monomial basis $\{x^{j-1}\}_{j=1}^r$. If the solution is written as
$p_{r-1}(x) = \sum_{j = 0}^{r-1} c_j x^j$, 
the interpolation conditions \eqref{eq:interpolationcondition} yield the system of equations
\[ \mu_i = \sum_{j = 0}^{r-1} c_j \frac{1}{j+1} (\beta_i^{j+1} - \alpha_i^{j+1}), \quad i \in \{1, \ldots, r\}. \]
We can write this more compactly using the Vandermonde matrix $\mathbf{V}^{(M)}(\mathcal{S}) \in \R^{r \times r}$ with the entries
\begin{equation} \label{eq:VandermMatrixMonomial}
\mathbf{V}^{(M)}_{i,j}(\mathcal{S})  = \frac{1}{j} (\beta_i^{j} - \alpha_i^{j}), \quad i,j \in \{1, \ldots, r\}.
\end{equation}
Then, the expansion coefficients $c_j$ of $p_{r-1}$ are obtained as the solution of the linear system
\begin{equation} \label{eq:VandermMonomial}
 \mathbf{V}^{(M)}(\mathcal{S}) \begin{pmatrix}
	c_0 \\ \vdots \\ c_{r-1}
\end{pmatrix} = \begin{pmatrix}
	\mu_1 \\ \vdots \\ \mu_r
\end{pmatrix}.
\end{equation}
The existence and uniqueness of the interpolation problem \eqref{eq:interpolationcondition} is determined by the invertibility of $\mathbf{V}^{(M)}(\mathcal{S})$ \cite[Lemma 2.2.1]{Davis75}. If $\mathbf{V}^{(M)}(\mathcal{S})$ is invertible, we call the set $\mathcal{S} = \{s_1, \ldots, s_r\}$ \emph{unisolvent} for the polynomial space $\mathbb{P}_{r-1}$.  
Compared to the nodal setting, the characterization of unisolvent sets $\mathcal{S}$ is, in general, more complex. We will study this issue more profoundly in the next section. 

\begin{remark} Although the monomial basis is very useful for theoretical purposes it has some drawbacks in practical computations. The Vandermonde matrix $\mathbf{V}^{(M)}(\mathcal{S})$ gets highly ill-conditioned already for small orders $r$. This fact typically leads to numerical instability in the calculation of the expansion coefficients $c_j$ and favors the selection of other basis systems for the space $\mathbb{P}_{r-1}$.  
\end{remark}

\subsection{Chebyshev polynomials of second kind}
For computational purposes, a more suitable basis for $\mathbb{P}_{r-1}$ is given by the Chebyshev polynomials $\{U_{j-1}\}_{j=1}^r$ of the second kind:
\[U_{j-1}(x) = \frac{\sin( j \arccos x )}{ \sin( \arccos x)}, \quad j \in \N.\]
For a solution with the expansion
$p_{r-1}(x) = \sum_{j = 0}^{r-1} a_j U_j(x)$
the interpolation conditions \eqref{eq:interpolationcondition} can be written as
\begin{align} \label{eq:interpcondChebyshev}
\mu_i &= \int_{\alpha_i}^{\beta_i} \sum_{j = 0}^{r-1} a_j U_j(x) \de x =  \sum_{j = 0}^{r-1} a_j \int_{a_i}^{b_i}U_j(x) \de x = \sum_{j = 0}^{r-1} a_j \frac{1}{j+1}( T_{j+1}(\beta_i) - T_{j+1}(\alpha_i))\\ &= \sum_{j = 1}^{r} a_{j-1} \frac{2}{j} \sin \left( j \frac{\arccos \alpha_i + \arccos \beta_i}{2}\right) \sin \left(j\frac{\arccos \alpha_i - \arccos \beta_i}{2}\right), \notag
\end{align} 
where $T_{j}(x) = \cos (j \arccos x)$ denote the Chebyshev polynomial of the first kind of degree $j$. 
In this derivation, we used the well-known fact that the derivative of the polynomial $T_{j}$ corresponds to $j U_{j-1}$. Now, if we define the Vandermonde matrix $\mathbf{V}^{(U)}(\mathcal{S}) \in \R^{r \times r}$ using the entries   
\begin{equation} \label{eq:VandermMatrixChebyshev}
\mathbf{V}^{(U)}_{i,j}(\mathcal{S}) =  \frac{2}{j} \sin \left( j \frac{\arccos \alpha_i + \arccos \beta_i}{2}\right) \sin \left(j\frac{\arccos \alpha_i - \arccos \beta_i}{2}\right), \quad i,j \in \{1, \ldots, r\},
\end{equation}
we obtain the expansion coefficients $a_j$ by calculating the solution of the linear system
\begin{equation} \label{eq:VandermSystemChebyshev}
 \mathbf{V}^{(U)}(\mathcal{S}) \begin{pmatrix}
	a_0 \\ \vdots \\ a_{r-1}
\end{pmatrix} = \begin{pmatrix}
	\mu_1 \\ \vdots \\ \mu_r
\end{pmatrix}.
\end{equation}
We use the Chebyshev system  $\{U_{j-1}\}_{j=1}^r$ and the Vandermonde matrix $\mathbf{V}^{(U)}(\mathcal{S})$ to calculate the interpolating polynomials $p_{r-1}$. The entire numerical procedure is summarized in Algorithm \ref{alg:1}.

\begin{algorithm}
\caption{Polynomial interpolation for function averages on interval segments}
\label{alg:1}
 			\begin{tabular}{p{15cm}*{1}{c}}
 				\vskip 0.01 cm
 				{\fontfamily{pcr} \selectfont INPUT:} Set of segments $\mathcal{S}=\{s_i = [\alpha_i, \beta_i], \hskip 0.1cm i \in \{1,\ldots,r\}\}$, 						\vskip 0.08 cm
 				\hskip 1.4cm a corresponding set of average values $\{\mu_i = \mu(f,s_i),\hskip 0.1cm i \in \{1,\ldots,r\} \}$,
 				\vskip 0.08 cm \hskip 1.4 cm and the desired evaluation point(s) $x \in I$.
 				\vskip 0.12 cm
 				{\fontfamily{pcr} \selectfont OUTPUT:} Polynomial interpolant $p_{r-1}(x)$. 
 				\vskip 0.24 cm
 				{\fontfamily{pcr} \selectfont Step 1:}
 				Using the segments $s_i = [\alpha_i,\beta_i]$, calculate the Vandermonde matrix $\mathbf{V}^{(U)}(\mathcal{S})$ in \eqref{eq:VandermMatrixChebyshev}.
 				\vskip 0.12 cm
 				{\fontfamily{pcr} \selectfont Step 2:}
 				Using the data $\{\mu_i\}_{i=1}^r$, solve the linear system \eqref{eq:VandermSystemChebyshev} to obtain the coefficients $\{a_j\}_{j=0}^{r-1}$.
 				\vskip 0.12 cm
 				{\fontfamily{pcr} \selectfont Step 3:}
 				Evaluate the polynomial interpolant $\ds p_{r-1}(x) = \sum_{j=0}^{r-1} a_j U_j(x)$.
 				\vskip 0.12 cm
 				\hskip 1.6 cm (The polynomials $U_j(x)$ can be evaluated be their three-term recurrence relation)
 				\\[\smallskipamount]
 			\end{tabular}
 \end{algorithm}
 
\subsection{Lagrange basis}

Lagrange polynomials are widely used tools both in interpolation theory and in finite element methods to represent interpolating polynomials in such a way that the given data information can be inserted directly. We recapitulate first the case of nodal interpolation.

\subsubsection{Nodal Lagrange basis}

In nodal interpolation, data is provided via function values $f(\xi_i)$ at $r$ distinct nodes $ -1 \leq \xi_1 < \xi_2 < \ldots < \xi_r \leq 1 $ in $I$. The Lagrange polynomials of degree $ r-1 $ form a basis $ \{ \ell_{\xi_j} \}_{j=1}^r $ for the space $ \mathbb{P}_{r-1}$ that is uniquely determined by the conditions
$ \ell_{\xi_j} (\xi_i) = \delta_{ij}$ for $i,j \in \{1, \ldots, r\}$. Because of this property the interpolation polynomial can be written as
\begin{equation} \label{eq:interpLagrbasisNodal}
     p_{r-1}(x) = \sum_{j=1}^r f(\xi_j) \ell_{\xi_j} (x).
\end{equation}
\vspace{-.2cm}
\begin{remark} \label{rmk:variouslagrangebasis}
An explicit formula for the $j$-th Lagrange polynomial is given by
\begin{equation} \label{eq:explicitformulaLagrbasis}
    \ell_{\xi_j}(x) = \prod_{\substack{i=1 \\ i\ne j}}^r \frac{x-\xi_i}{\xi_j - \xi_i} .
\end{equation}
We have the identity $\sum_{j=1}^r \ell_{\xi_j}(x) = 1$ and, hence, the relationship $\sum_{j=1}^r \ell_{\xi_j}'(x) = 0$ for the derivatives. On the other hand, any subset of $ r-1 $ elements out of the system $ \{ \ell_{\xi_j}' (x) \}_{j=1}^r $ is a basis for $ \mathbb{P}_{r-2}$. Further, the logarithmic derivative gives the following expression for $ \ell_{\xi_j}'(x)$
 \begin{equation} \label{eq:explicitformulaLagrDerivative}
 \ell_{\xi_j}' (x) = \ell_{\xi_j} (x) \sum_{\substack{i = 1 \\ i \ne j}}^r \frac{1}{x-\xi_{i}} .
\end{equation}   
\end{remark}

\subsubsection{Lagrange bases for interval segments}
As in the nodal setting, we can define a Lagrange basis for function data on interval segments as the system of polynomials $\{\ell_{s_j}\}_{j = 1}^r$ of degree $r-1$ that satisfy the conditions
\begin{equation} \label{eq:Lagrcondition}
\int_{s_i} \ell_{s_j}(x) \de x = \delta_{ij}, \quad i,j \in \{1, \ldots, r\}. 
\end{equation} 
The interpolation polynomial in the generalized segmental setting can then be written as
\begin{equation} \label{eq:interpLagrbasisSegmental}
     p_{r-1}(x) = \sum_{j=1}^r \mu(f,s_j) \ell_{s_j}(x) .
\end{equation}
In comparison to the nodal case, there are no general explicit formulas for the Lagrange polynomials $\ell_{s_j}(x)$. In the particular settings \ref{itm:one} and \ref{itm:three}, we will be able to derive an explicit expression for the polynomials $\ell_{s_j}(x)$, avoiding in this way the inversion of a (possibily ill-conditioned) Vandermonde matrix. The Lagrange polynomials $\ell_{s_j}(x)$ can nevertheless be calculated using Algorithm \ref{alg:1} by using the data vectors $\mu_i = \delta_{ij}$, $i \in \{1, \ldots,r\}$ as right hand sides in \eqref{eq:VandermSystemChebyshev}.

\section{Existence and uniqueness} \label{sect:uniqueness}

Before studying the scenarios \ref{itm:one}, \ref{itm:two} and \ref{itm:three} in more detail, we state a general result on the uniqueness of the polynomial interpolant in case that the segments in $\mathcal{S}$ do not overlap.

\begin{proposition} \label{prop:uniquenessnonoverlapping}
If the segments in $\mathcal{S} = \{s_1, \ldots, s_r\}$ are non-overlapping in the sense that $|s_i \cap s_j| = 0$ for $i \neq j$, then the set $\mathcal{S}$ is unisolvent for $\mathbb{P}_{r-1}$.
\end{proposition}

\begin{proof}
The set $\mathcal{S}$ is unisolvent for $\mathbb{P}_{r-1}$ if, for each $ p_{r-1} \in \mathbb{P}_{r-1} $, the $ r $ conditions 
\begin{equation} \label{eq:condunisolvence}
    \int_{s_i} p_{r-1}(x) \de x = 0, \quad i = \{ 1, \ldots, r \},
\end{equation}
imply that $ p_{r-1} (x) = 0 $ for each $ x $. Consider $ p_{r-1} \in \mathbb{P}_{r-1} $ and assume that \eqref{eq:condunisolvence} holds. Then, by the Lagrange Theorem, for each $ s_i = [\alpha_i, \beta_i] $ there exists a point $ \xi_i $ such that $ \alpha_i < \xi_i < \beta_i $ where $ p_{r-1}(\xi_i) = 0 $. Since the intervals intersect, at most, in their endpoints, all the nodes $ \{ \xi_i \}_{i=1}^r $ are pairwise distinct, hence unisolvent for $ \mathbb{P}_{r-1} $. It follows that $ p_{r-1} = 0 $ everywhere and hence that the set of segments $ \mathcal{S} $ is unisolvent for $ \mathbb{P}_{r-1} $ as well. 
\end{proof}

\subsection{The case \ref{itm:one}: an explicit cardinal basis for chains of segments} \label{sect:explCB}

With $r+1$ nodes $-1 = \xi_0 < \xi_1 < \cdots < \xi_r = 1$ given, the case \ref{itm:one} considers polynomial interpolation using the $r$ interval segments $s_i = [\xi_{i-1}, \xi_{i}]$. We know from Proposition \ref{prop:uniquenessnonoverlapping} that these segment sets are unisolvent. We want to derive now an explicit formula for the polynomials $ \{ \ell_{s_1}, \ldots, \ell_{s_r} \} $ such that 
$$ \int_{s_i} \ell_{s_j}(x) \de x = \delta_{ij}.$$
 The simplification to concatenated segments in \ref{itm:one} allows to express the polynomials $\ell_{s_i}$ in terms of linear combinations of the derivatives $ \{\ell_{\xi_j}' \}_{j=0}^r$ of the nodal Lagrange polynomials $ \ell_{\xi_j}$'s defined upon the nodes $\{\xi_0, \ldots, \xi_r\}$. 
More precisely, the system $\{ \ell_{\xi_j}' \}_{j=0}^r $ generates, but is not a basis of $\mathbb{P}_{r-1}$, since we have $ r+1 > r $ elements, see Remark \ref{rmk:variouslagrangebasis}. 
Nevertheless, since any subset of $r$ elements of $\{ \ell_{\xi_j}' \}_{j=0}^r $ defines a basis for $\mathbb{P}_{r-1}$, there are coefficients $ a_0^{(j)}, \ldots, a_r^{(j)} $ (not unique) such that 
$$ \left\{ \ell_{s_1}(x) = \sum_{j=0}^r a_j^{(1)} \ell_{\xi_j}'(x),\; \ldots\;,\; \ell_{s_r}(x) = \sum_{j=0}^r a_j^{(r)} \ell_{\xi_j}'(x) \right\} $$
is the desired basis. We can thus expand
\begin{align*}
	\delta_{ij} &= \int_{s_i} \ell_{s_j}(x) \de x =  \int_{s_i} \sum_{k=0}^r a_k^{(j)} \ell_{\xi_k}'(x) \de x = \sum_{k=0}^r a_k^{(j)} \int_{s_i} \ell_{\xi_k}'(x) \de x \\ & = \sum_{k=0}^r a_k^{(j)} \int_{\xi_{i-1}}^{\xi_{i}} \ell_{\xi_k}'(x) \de x = \sum_{k=0}^r a_k^{(j)} \left(\ell_{\xi_k} (\xi_{i}) - \ell_{\xi_k} (\xi_{i-1}) \right),
\end{align*}
where the last equality is granted by the fundamental theorem of calculus. Hence
$$ 0 = \sum_{k=0}^r a_k^{(j)} \ell_{\xi_k} (\xi_{i}) - \sum_{k=0}^r a_k^{(j)} \ell_{\xi_k} (\xi_{i-1}) \quad \text{if } i \ne j $$
and
$$ 1 = \sum_{k=0}^r a_k^{(j)} \ell_{\xi_k} (\xi_{i}) - \sum_{k=0}^r a_k^{(j)} \ell_{\xi_k} (\xi_{i-1}) \quad \text{if } i = j .$$
Now, plugging in the fact that $ \{\ell_{\xi_i}\}_{i=0}^r $ is the Lagrange basis with respect to the $ \xi_i $'s, we get
$$ a_{i}^{(j)} = a_{i-1}^{(j)} \quad \text{if } i \ne j, $$
and
$$ a_{j}^{(j)} = a_{j-1}^{(j)} + 1 .$$ 
Observe that the solutions of these systems have one free parameter, which is a consequence of the fact that $ r + 1 $ points define $ r $ intervals. By imposing $a_0^{(j)} = 0$, one solution that characterizes the Lagrange polynomials $\ell_{s_j}$, $j \in \{1, \ldots, r\}$ is given as
\begin{equation} \label{eq:lagrange1form}
     \ell_{s_j}(x) = \sum_{k = j}^r \ell_{\xi_k}'(x) .
\end{equation}
In fact, one can directly check that the right hand side of this identity satisfies the Lagrange conditions \eqref{eq:Lagrcondition}. Furthermore, the triangular expansion \eqref{eq:lagrange1form} in terms of the basis $\{\ell_{s_k}'\}_{k=1}^r$ immediately implies that also the polynomials $\{\ell_{s_j}\}_{j=1}^r$ form a basis for $\mathbb{P}_{r-1}$. We have thus proved the following.

\begin{proposition} \label{prop:lagrbasis}
    In the case \ref{itm:one} of concatenated segments $s_i = [\xi_{i-1}, \xi_i]$, $i\in \{1, \ldots r\}$, the segments $\{s_1, \ldots, s_r\}$ are unisolvent for $ \mathbb{P}_{r-1}$ and the polynomials 
    $$ \ell_{s_j}(x) = \sum_{k=j}^r \ell_{\xi_k}'(x), \quad j \in \{1, \ldots, r\},$$
    form a cardinal Lagrange basis of $ \mathbb{P}_{r-1}$.
\end{proposition}

\begin{figure}[htbp]
    \centering
    \includegraphics[width=7.5cm]{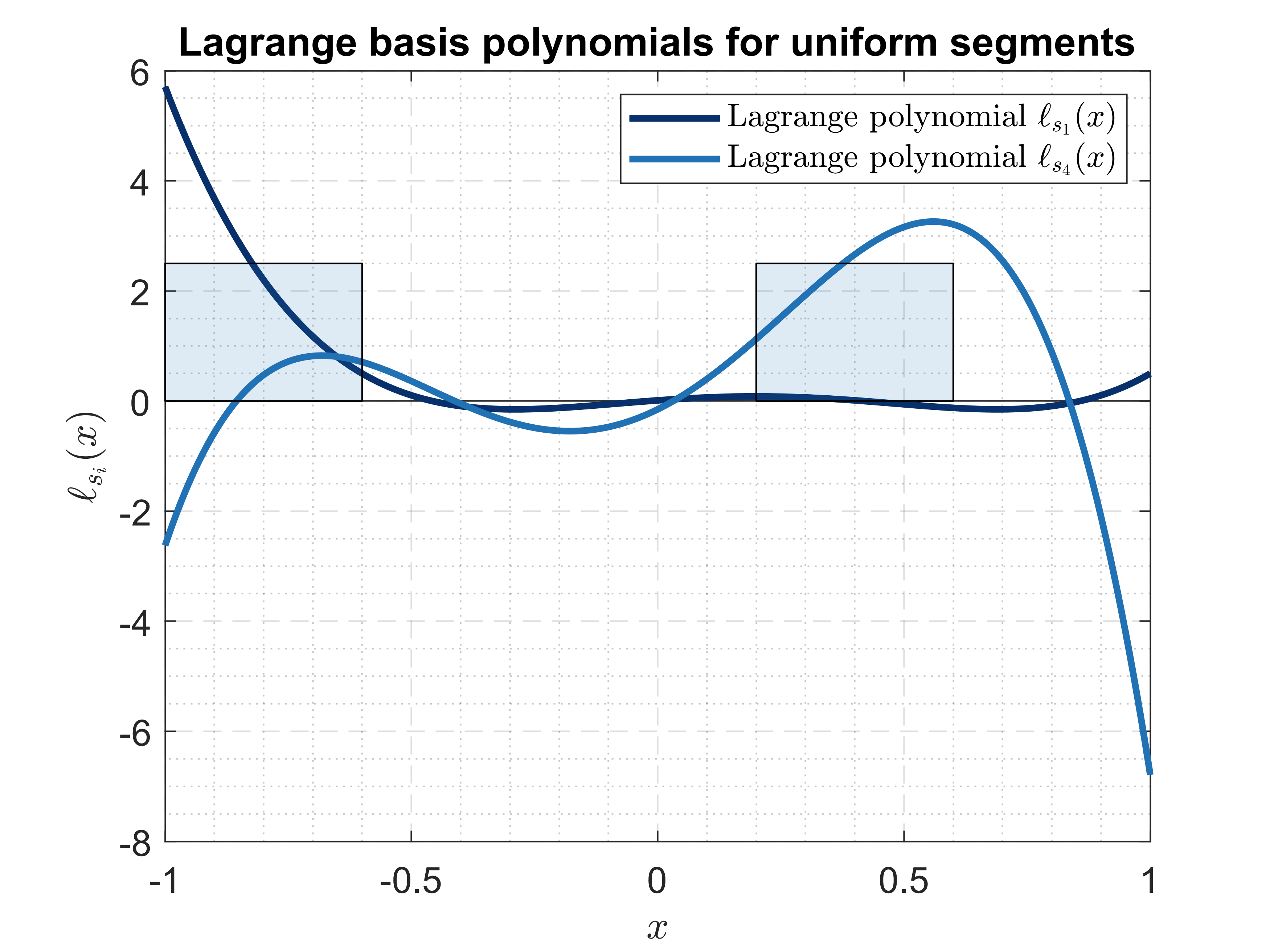}
    \includegraphics[width=7.5cm]{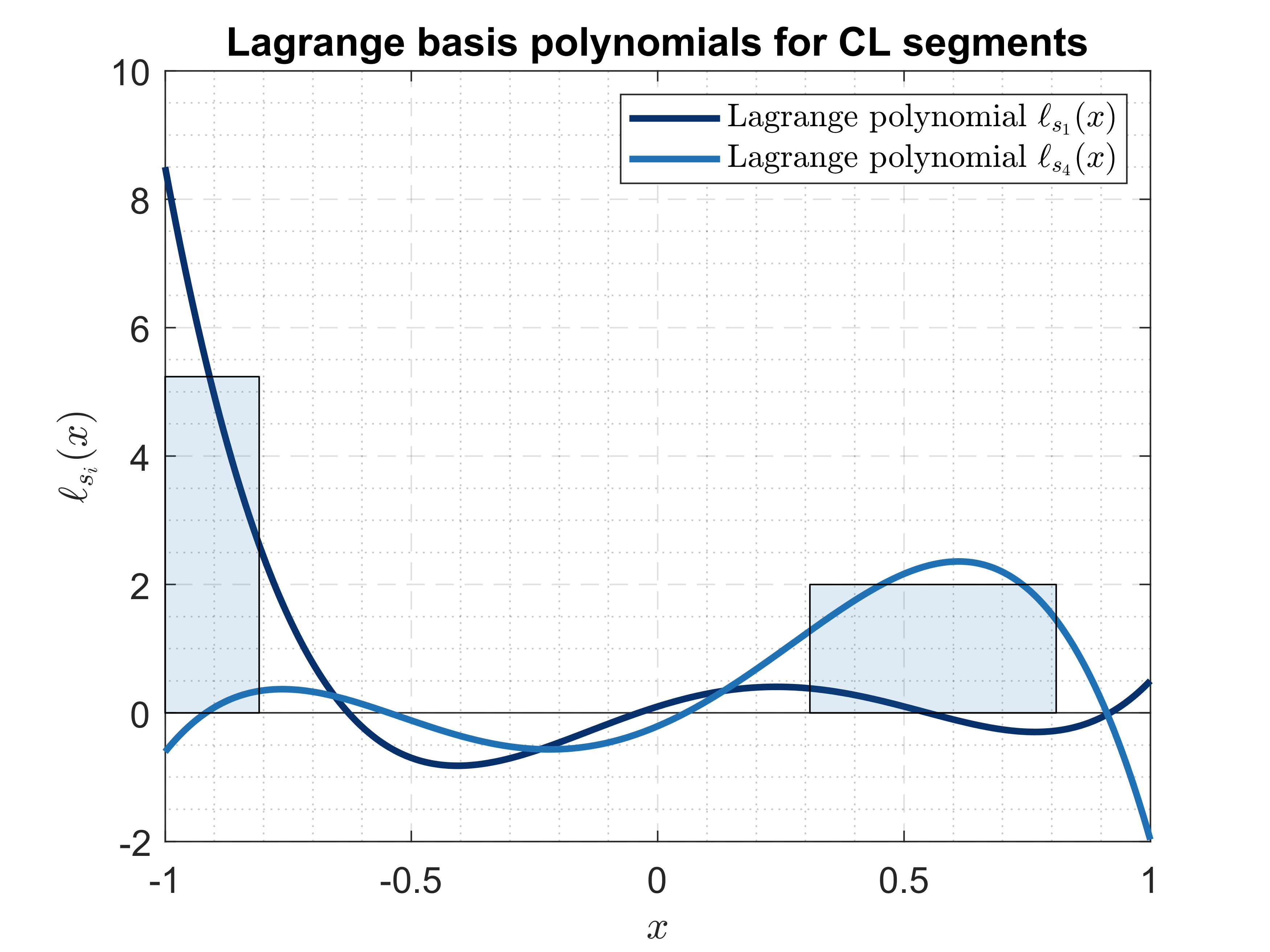}
    \vspace{-.7cm}
    \caption{Lagrange basis polynomials $\ell_{s_1}$ and $\ell_{s_4}$ for $r = 5$ interval segments in the class \ref{itm:one}. The integral of the polynomials on the segments corresponds to the area of the visualized rectangles. Left: Lagrange polynomials for equidistant segments. Right: Lagrange polynomials for Chebyshev-Lobatto segments.}
    \label{fig:lagrangeforms}
\end{figure}

\begin{remark}
Proposition \ref{prop:lagrbasis} and Eq. \eqref{eq:explicitformulaLagrDerivative} allow to calculate the polynomials $\ell_{s_j}$ explicitly. A depiction of some Lagrange basis polynomials is provided in Fig. \ref{fig:lagrangeforms}. We saw in the derivation of Proposition \ref{prop:lagrbasis} that the representation of  $\ell_{s_j}$ in terms of the generating system $\{\ell_{\xi_k}'\}_{k=0}^r$ is not unique. If we use the identity $\sum_{k=0}^r \ell_{\xi_k}'(x) = 0$ seen in Remark \ref{rmk:variouslagrangebasis}, we can also write
    $$ \sum_{k=0}^{j-1} \ell_{\xi_k}'(x) = - \sum_{k=j}^r \ell_{\xi_k}'(x).$$
Hence, we can write the $j$-th element of the Lagrange basis in Proposition \ref{prop:lagrbasis} equivalently as
    $$  \ell_{s_j}(x) = - \sum_{k=0}^{j-1} \ell_{\xi_k}'(x).$$
    \end{remark}
    
\subsection{The case \ref{itm:two}: uniqueness for proper choices of the arc-length}
We consider now sets of segments with constant arc-length. In particular, we assume that the borders of the $r$ intervals $s_i$ are given as $\alpha_i = \cos(\tau_i + \rho)$ and $\beta_i = \cos(\tau_i - \rho)$
with $0 < \tau_1 < \cdots < \tau_r < \pi$ denoting the arc-midpoints of the segments and $0 < \rho< \pi$ the respective constant arc-radius. For this type of segments, the representation of the interpolation polynomial in terms of the Chebyshev basis $\{U_{j-1}\}_{j=1}^{r}$ is particularly useful. If we insert these segments in the interpolation condition \eqref{eq:interpcondChebyshev} based on the basis polynomials $U_j$, we get the simplification
\begin{equation*} 
\mu_i = \sum_{j = 1}^{r} \frac{2 a_{j-1} }{j} \sin \left( j \tau_i \right) \sin \left(j \rho \right).
\end{equation*}
If we devide both sides by the length $|s_i| = 2 \sin(\tau_i) \sin(\rho)$ of the segments, we can conclude 
\begin{equation} \label{eq:interpcondC2}
\frac{\mu_i}{|s_i|} = \sum_{j = 1}^{r} a_{j-1} \frac{\sin \left(j \rho \right)}{j \sin(\rho)} 
\frac{\sin \left( j \tau_i \right)}{\sin(\tau_i)} = \sum_{j = 1}^{r} a_{j-1} \frac{\sin \left(j \rho \right)}{j \sin(\rho)} U_{j-1}(\cos \tau_j).
\end{equation}
We can interpret the term on the right hand side of \eqref{eq:interpcondC2} as a polynomial $q_{r-1}(x) = \sum_{j=0}^{r-1} b_{j} U_j(x)$ of degree $r-1$ evaluated at the nodes $\cos \tau_j$. This polynomial  interpolates the $r$ values $\frac{\mu_i}{|s_i|}$ at the distinct points $\cos \tau_i \in I$ and is therefore uniquely determined. Between the expansion coefficients $\{b_j\}_{j=0}^{r-1}$ of the polynomial $q_{r-1}$ in the basis $\{U_j\}_{j=0}^{r-1}$ and the expansion coefficients $\{a_j\}_{j=0}^{r-1}$ of the polynomial $p_{r-1}$ we have the relation
\begin{equation} \label{eq:relationcoefficentsC2}
b_j = a_j \frac{\sin \left((j+1) \rho \right)}{(j+1) \sin(\rho)}, \quad j \in \{0, \ldots, r-1\}.
\end{equation} 
The coefficients $a_j$ can be recovered from the coefficients $b_j$ if and only if for the arc-radius $0<\rho<\pi$ we have $\rho \notin \{\frac{k \pi}{j} \: |\: j,k \in \{1, \ldots, r\}, \; j > k\}$.  We thus get the following result on the existence and uniqueness in the case \ref{itm:two}:

\begin{proposition} \label{prop:uniquenessC2}
    In the case \ref{itm:two} of segments $s_i = [\cos(\tau_i + \rho), \cos(\tau_i - \rho)]$, $i\in \{1, \ldots r\}$, with constant arc-radius $0 < \rho < \pi$, the segments $\{s_1, \ldots, s_r\}$ are unisolvent
  for $\mathbb{P}_{r-1}$ if and only if 
  $$\rho \notin \left\{\frac{k \pi}{j} \: |\: j,k \in \{1, \ldots, r\}, \; j > k \right\}.$$  
In case of unisolvence, the interpolating polynomial $p_{r-1}$ can be obtained by calculating the nodal interpolant $q_{r-1}$ for the nodes $\{\cos \tau_i\}_{i=1}^r$ and using the relation \eqref{eq:relationcoefficentsC2} between the expansion coefficients of $p_{r-1}$ and $q_{r-1}$. 
\end{proposition}
	
\begin{remark}
In the limit $\rho \to 0$ (or also $\rho \to \pi$), we get the identity $a_j = b_j$. This limiting case corresponds exactly to the case of nodal interpolation in the points $\cos \tau_i$, $i \in \{1, \ldots r\}$. Nodal interpolation can therefore be regarded as a limiting case of interpolation on segments when the arc-length of the segments tends to zero.
\end{remark}

\subsection{The case \ref{itm:three}: a Lagrange basis for segments with the same left endpoint} \label{sect:explLB}  
With a proper scaling of the intervals and a rearrangement of the available data the case \ref{itm:three} can be mapped to the case \ref{itm:one}, the only conceptual difference being a different structuring of the segments $s_i$. The interesting aspects of the case \ref{itm:three} are that the segments are overlapping and that the cardinal basis can be expressed in a very simple form. 

The segments in \ref{itm:three} have the particular form $s_i = [\alpha,\beta_i]$ with $-1 \leq \alpha < \beta_1 < \cdots < \beta_r \leq 1 $. We consider the nodal Lagrange basis $\{\ell_{\alpha}, \ell_{\beta_1}, \cdots \ell_{\beta_r}\}$ with respect to the node set $\{\alpha, \beta_1, \ldots , \beta_r\}$ that forms a basis for the space $\mathbb{P}_r$. Then the $r$ derivatives $\ell_{\beta_i}'(x)$ satisfy
\[ \int_{\alpha}^{\beta_j}\ell_{\beta_i}'(x) \de x = \ell_{\beta_i}(\beta_j) - \ell_{\beta_i}(\alpha) = \ell_{\beta_i}(\beta_j) = \delta_{ij}, \quad i,j \in \{1, \ldots, r\}.\]    
This is exactly the Lagrange condition for the segments $s_i = [\alpha,\beta_i]$. We can therefore conclude that the Lagrange basis with respect to the segments $\{s_1, \ldots, s_r\}$ can be written as $\ell_{s_i}(x) = \ell_{\beta_i}'(x)$, $i \in \{1, \ldots, r\}$. As the derivatives $\{\ell_{\beta_j}'\}_{j = 1}^r$ form a basis for the space $\mathbb{P}_{r-1}$ (see Remark \ref{rmk:variouslagrangebasis}), we can also conclude that the set $\{s_1, \ldots, s_r\}$ is unisolvent. We thus get the following result. 

\begin{proposition} \label{prop:lagrbasisLB}
    In the case \ref{itm:three} of segments of the form $s_i = [\alpha, \beta_i]$, $-1 \leq \alpha < \beta_1 < \cdots < \beta_r \leq 1 $, the set $\{s_1, \ldots, s_r\}$ is unisolvent for $ \mathbb{P}_{r-1}$ and the polynomials 
    $$ \ell_{s_j}(x) = \ell_{\beta_j}'(x), \quad j \in \{1, \ldots, r\},$$
    form a cardinal Lagrange basis of $\mathbb{P}_{r-1}$ with respect to the set $\{s_1, \ldots, s_r\}$. 
\end{proposition}

Note that, similarly as discussed in Section \ref{sect:explCB}, the Lagrange polynomials $\ell_{s_j}(x)$ can be expressed in alternative ways if one includes also the polynomial $\ell_{\alpha}'(x)$ in the description. Nevertheless, the characterization given in Proposition \ref{prop:lagrbasisLB} is particularly simple and allows to compute $\ell_{s_j}(x)$ explicitly using the formula \eqref{eq:explicitformulaLagrDerivative} for the derivatives $\ell_{\beta_j}'(x)$.

\section{Numerical conditioning and the Lebesgue constant} \label{sect:Leb}

\subsection{The Lebesgue constant in the nodal setting}

Given a set $\mathcal{X} = \{\xi_1, \ldots, \xi_r\} \subset I$ of distinct nodes, the Lebesgue constant with respect to the nodal interpolation on $\mathcal{X}$ is defined as
\begin{equation*}
    \Lambda_r(\mathcal{X}) = \sup_{x \in I} \sum_{i=1}^r \vert \ell_{\xi_i} (x) \vert . 
\end{equation*}
Further, if we define the (nodal) interpolation operator $\Pi_r(\mathcal{X}): C(I) \to \mathbb{P}_{r-1}$ as the operator
\begin{align} \label{eq:interpoperatorNodal}
	\Pi_{r}(\mathcal{X}) f (x) = \sum_{i = 1}^r f(\xi_i) \ell_{\xi_i} (x) = p_{r-1}(x),
\end{align} 
that projects a continuous function $f$ onto the interpolating polynomial $p_{r-1}$, we have the identity 
\[ \Lambda_r(\mathcal{X}) = \|\Pi_{r}(\mathcal{X})\|_{\mathrm{op}} = \sup_{\|f\| \leq 1} \|\Pi_{r}(\mathcal{X}) f\|,\]
where $\|\cdot \|_{\mathrm{op}}$ denotes the operator norm of $\Pi_{r}(\mathcal{X})$ with respect to the uniform norm $\|\cdot\|$, see \cite{Trefethen}. The Lebesgue constant describes the propagation of errors in the interpolation process and can therefore be regarded as the numerical conditioning for the solution of the interpolation problem. We extend this concept now for interpolation with regard to averages on segments.

\subsection{The Lebesgue constant for interpolation on segments} \label{sect:Lebconstinterpsegment}

We define the Lebesgue constant associated with polynomial interpolation on a set of segments $\mathcal{S} = \{s_1, \ldots, s_r\}$ as 
\begin{equation} \label{eq:genLeb}
\Lambda_r (\mathcal{S}) \doteq \sup_{-1 \leq a < b \leq 1} \frac{1}{b-a} \sum_{i = 1}^r |s_i | \left\vert \int_a^b \ell_{s_i}(x) \de x  \right\vert = \sup_{x \in I} \sum_{i = 1}^r |s_i | |\ell_{s_i}(x)|. 
\end{equation}
Both representations in \eqref{eq:genLeb} will be helpful for us. As in the nodal setting, it is natural to consider the linear interpolation operator $\Pi_{r} \doteq \Pi_{r}(\mathcal{S})$, now given as
\begin{align} \label{eq:interpoperatorexpl}
	\Pi_{r}(\mathcal{S}): \quad L_{\infty}(I) &\to \mathbb{P}_{r-1}, \quad 
	f \mapsto \sum_{i=1}^r \mu(f,s_i) \ell_{s_i}(x) = p_{r-1},
\end{align} 
that maps a bounded function $f$ to its interpolation polynomial $p_{r-1}$. 
For interpolation on general segments, the Lebesgue constant \eqref{eq:genLeb} gives an upper bound for the operator norm of the interpolation operator $\Pi_{r}(\mathcal{S})$ in \eqref{eq:interpoperatorexpl}. Equality does only hold in particular cases. 

\begin{theorem} \label{thm:operatornorm}
    Let $\mathcal{S} = \{s_1, \ldots, s_r\}$ be unisolvent for $\mathbb{P}_{r-1}$ with $ \Pi_{r}(\mathcal{S}): L_{\infty}(I) \to \mathbb{P}_{r-1}$
    being the corresponding interpolation operator. Then
    \begin{equation} \label{eq:upperboundnorm}
    \Vert \Pi_{r}(\mathcal{S}) \Vert_{\mathrm{op}} \leq \Lambda_r (\mathcal{S}).
    \end{equation}     
    If the segments in $\mathcal{S}$ are pairwise non-overlapping in the sense that $|s_i\cap s_j| = 0$ for $i \neq j$, then equality
    $ \Vert \Pi_{r}(\mathcal{S}) \Vert_{\mathrm{op}} = \Lambda_r (\mathcal{S})$ holds in \eqref{eq:upperboundnorm}.
\end{theorem}

\begin{proof}
We first show $ \Vert \Pi_r \Vert_{\mathrm{op}} \leq \Lambda_r \doteq\Lambda_r(\mathcal{S})$. Untangling the given definitions, we find
	\begin{align*}
\Vert \Pi_r \Vert_{\mathrm{op}} & \doteq \sup_{\Vert f \Vert = 1} \Vert \Pi_r f \Vert = \sup_{\Vert f \Vert = 1} \sup_{x \in I} \left\vert \Pi_r f(x)  \right\vert 
		 = \sup_{\Vert f \Vert = 1} \sup_{x \in I}  \left\vert \left( \sum_{i=1}^r \mu(f,s_i) \ell_{s_i}(x) \right) \right\vert \\
		   &= \sup_{\Vert f \Vert = 1} \sup_{x \in I}  \left\vert \sum_{i = 1}^r |s_i| \frac{\mu(f,s_i)}{|s_i|} \ell_{s_i}(x)  \right\vert. 
	\end{align*}
	Since $ \frac{1}{|s_i|} |\mu(f,s_i)| = \frac{1}{|s_i|}| \int_{s_i} f(x) \de x | \leq \Vert f \Vert $, 
	we get
	\begin{align*}
		\Vert \Pi_r \Vert_{\mathrm{op}}
		& \leq \sup_{\Vert f \Vert = 1} \sup_{x \in I} \Vert f \Vert  \sum_{i = 1}^r |s_i| \left\vert  \ell_{s_i}(x)  \right\vert 
		  = \sup_{x \in I} \sum_{i = 1}^r |s_i| \left\vert  \ell_{s_i}(x)  \right\vert = \Lambda_r  .
	\end{align*}
	We have hence proved \eqref{eq:upperboundnorm}. To prove the converse for non-overlapping segments $s_i$, we show that for each fixed $ 0 < \varepsilon < \Lambda_r $ there exists a continuous function $f_{\varepsilon}$ such that $ \Vert f_{\varepsilon}\Vert = 1 $ and 
	$$ \Vert \Pi_r f_{\varepsilon} \Vert \geq \Lambda_r - \varepsilon .$$  Letting $ \varepsilon \to 0 $ the claim will then follow.
	For each $ \varepsilon > 0 $, we consider for every $s_i$ an open subset $ U_{s_i}^{\varepsilon} $  such that $ | U_{s_i}^\varepsilon | \geq \left(1 - \frac{\varepsilon}{\Lambda_r} \right) |s_i| $.
	As the intervals $s_i$ are non-overlapping, we can define the subsets
	\begin{equation} \label{eq:defomegaell}
		B_{\varepsilon} (I) \doteq \left\{ f \in C(I) \ :  \ \Vert f \Vert = 1, \
		f \bigl|_{U_{s_i}^\varepsilon} = \pm 1,
		\ \mathrm{sgn} (f) \text{ is constant in $s_i$ for } i \in \{1, \ldots, r\} 
		\right\}.
	\end{equation}
	For every function $ f \in B_{\varepsilon} (I) $, we have
	\begin{equation} \label{eq:intomegaell}
		|\mu(f,s_i)| = \left\vert \int_{U_{s_i}^\varepsilon} f \de x   + \int_{{s_i} \setminus U_{s_i}^\varepsilon} f \de x  \right \vert = \left\vert \int_{U_{s_i}^\varepsilon} f \de x  \right\vert + \left\vert  \int_{{s_i} \setminus U_{s_i}^\varepsilon} f \de x  \right \vert = |U_{s_i}^\varepsilon| + \left\vert\int_{{s_i} \setminus U_{s_i}^\varepsilon} f \de x \right \vert \geq | U_{s_i}^\varepsilon | ,
	\end{equation}
	where the second equality is granted as we ask that $ \mathrm{sgn} (f) $ is constant in $ s_i $. We further have
	\begin{equation} \label{eq:minorazione}
		\Vert \Pi_r \Vert_{\mathrm{op}} \geq \sup_{f \in B_{\varepsilon} (I)} \Vert \Pi_r f \Vert.
	\end{equation}
	We next expand $ \Vert \Pi_r f \Vert $ as
	\begin{align*}
		\Vert \Pi_r \omega \Vert & = \sup_{x \in I} \left\vert \Pi_r f(x) \right\vert = \sup_{x \in I} \left\vert \sum_{i= 1}^r \mu(f,s_i) \ell_{s_i}(x)  \right\vert = \sup_{x \in I} \left\vert \sum_{i= 1}^r \mu(f,s_i) \, \mathrm{sgn} (\ell_{s_i}(x)) \left\vert \ell_{s_i}(x)  \right\vert\right\vert .
	\end{align*}
	By the definition of $B_{\varepsilon} (I)$, for each $ x \in I $ there exists $ f_\varepsilon \in B_{\varepsilon} (I) $ such that $ \mathrm{sgn} (\ell_{s_i}(x)) \mu(f_\varepsilon,s_i) \geq 0 $ for all $i \in \{1, \ldots, r\}$. Hence 
	$$  \mathrm{sgn} (\ell_{s_i}(x)) \mu(f_\varepsilon,s_i) = |\mu(f_\varepsilon,s_i)|, \quad i \in \{1, \ldots, r\}. $$
	For such $f_\varepsilon \in B_{\varepsilon} (I)$, the right hand side of \eqref{eq:minorazione} can be bounded from below by applying \eqref{eq:intomegaell} as
	\begin{align*}
		\Vert \Pi_r f_\varepsilon \Vert & = \sup_{x \in I}  \left\vert \sum_{i= 1}^r |\mu(f,s_i)| \, \left\vert \ell_{s_i}(x)  \right\vert\right\vert  
		  \geq \sup_{x \in I}  \left\vert \sum_{i= 1}^r |U_{s_i}^\varepsilon| \, \left\vert \ell_{s_i}(x)  \right\vert\right\vert  \\
	    & \geq \sup_{x \in I}  \left\vert \sum_{i= 1}^r \left(1- \frac{\varepsilon}{\Lambda_r}\right) |s_i| \, \left\vert \ell_{s_i}(x)  \right\vert\right\vert   = \left(1- \frac{\varepsilon}{\Lambda_r}\right) \sup_{x \in I} \left\vert \sum_{i= 1}^r  |s_i| \, \left\vert \ell_{s_i}(x)  \right\vert\right\vert  = \left(1- \frac{\varepsilon}{\Lambda_r}\right) \Lambda_r.
	\end{align*}
	Hence we have $ \Vert \Pi_r \Vert_{\mathrm{op}} \geq \Lambda_r - \varepsilon $. This, together with \eqref{eq:upperboundnorm}, implies that $ \Vert \Pi_r \Vert_{\mathrm{op}} = \Lambda_r $.
\end{proof}

\begin{remark}
For overlapping segments, we have in general only the upper estimate \eqref{eq:upperboundnorm}. In fact, in some cases the Lebesgue constant $\Lambda_r(\mathcal{S})$ can overestimate the operator norm $\Vert \Pi_r(\mathcal{S}) \Vert_{\mathrm{op}}$ quite considerably. In Fig. \ref{fig:LebConstC3}, the differences between the two values are visualized for the overlapping Chebyshev-Lobatto segments $\mathcal{S}^{\mathrm{CLO}}$ in the class \ref{itm:three}. While the operator norm $\Vert \Pi_r(\mathcal{S}^{\mathrm{CLO}}) \Vert_{\mathrm{op}}$ grows logarithmically in $r$ the growth of the constant $\Lambda_r(\mathcal{S}^{\mathrm{CLO}})$ is considerably faster. 

\end{remark}

\begin{figure}[htbp]
    \centering
    \includegraphics[width=7.5cm]{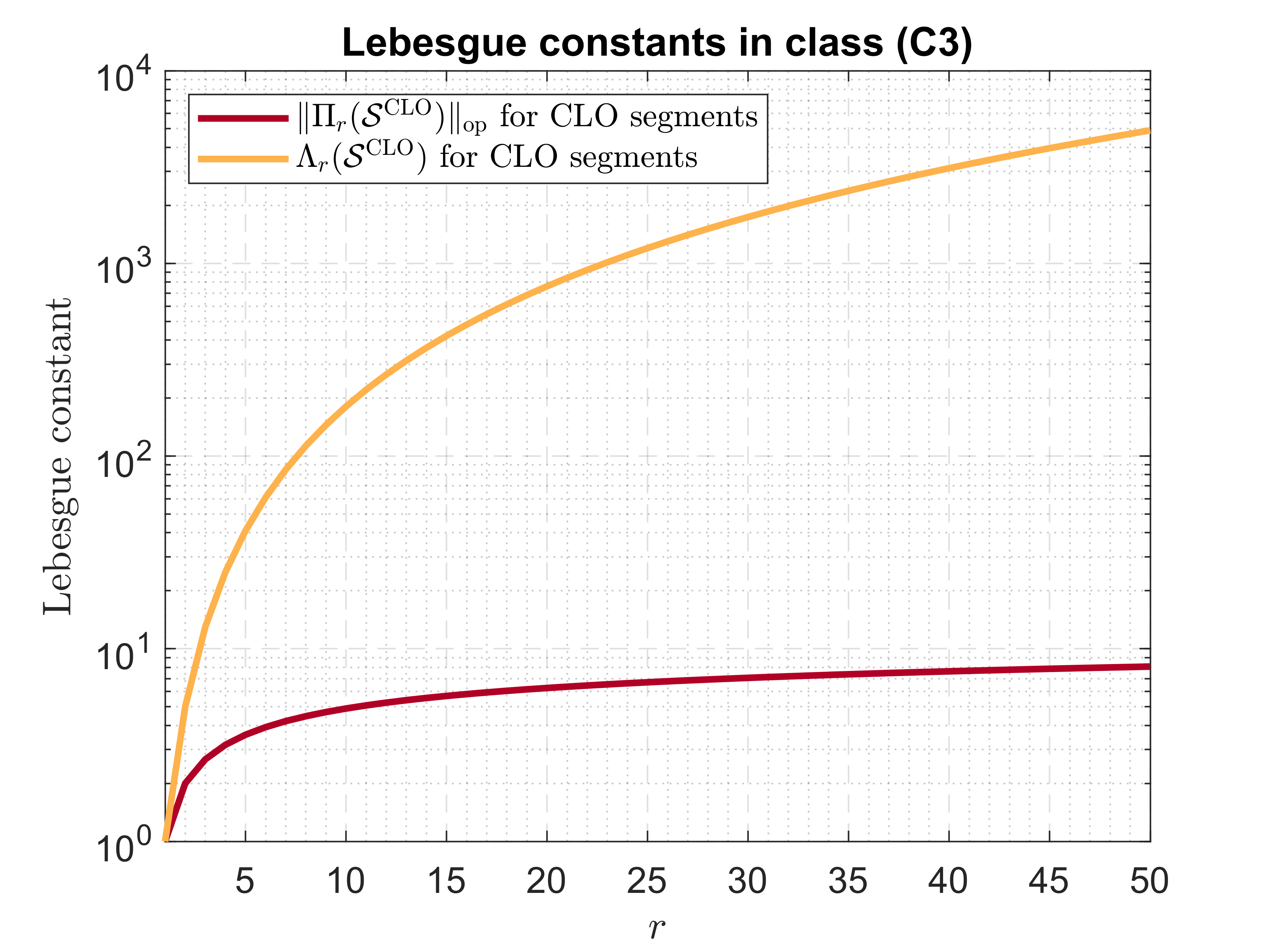}
    \vspace{-.4cm}
    \caption{Difference between the Lebesgue constant $\Lambda_r(\mathcal{S}^{\mathrm{CLO}})$ and the operator norm $\Vert \Pi_r (\mathcal{S}^{\mathrm{CLO}})\Vert_{\mathrm{op}}$ for the overlapping Chebyshev-Lobatto segments $\mathcal{S}^{\mathrm{CLO}}$ in the class \ref{itm:three}.  }
    \label{fig:LebConstC3}
\end{figure}

\subsection{Invariance of the Lebesgue constants}

A charming aspect of the nodal Lebesgue constant is that it depends only on the position of the nodes inside an interval $ I $ and not on the interval $ I $ itself \cite{Hestaven}. This feature is borne by the segmental Lebesgue constant \eqref{eq:genLeb} as well.

\begin{proposition} \label{prop:invarianceLeb}
	The segmental Lebesgue constant does not depend on the interval $ I $.
\end{proposition}

\begin{proof}
Let $ \varphi: I \to I' $ be the orientation-preserving affinity that maps $ I $ onto $I'$. By the change of variable, if $ \ell_{s_i}^I $ is the $i$-th element of this generalised Lagrange basis for $ I $, then $ \ell_{s_i}^{I'} = \frac{|s_i|}{|\varphi(s_i)|} \ell_{s_i}^I \circ \varphi^{-1} $ is the $i$-th element of the generalised Lagrange basis for $ I' $. A direct computation thus yields
		\begin{align*}
		\Lambda^{I'} (\mathcal{S}) & = \max_{y \in I'} \sum_i |\varphi(s_i) | \left\vert \ell_{s_i}^{I'} (y) \right\vert = \max_{y \in I'} \sum_i |\varphi(s_i) | \frac{|s_i|}{|\varphi(s_i)|} \left\vert \ell_{s_i}^I \circ \varphi^{-1} (y) \right\vert \\& = \max_{\varphi(x) \in I'} \sum_i |s_i | \left\vert \ell_{s_i}^I \circ \varphi^{-1} (\varphi(x)) \right\vert = \max_{x \in I} \sum_i |s_i | \left\vert \ell_{s_i}^I (x) \right\vert = \Lambda^{I} (\mathcal{S}) .
	\end{align*}
\end{proof}

Proposition \ref{prop:invarianceLeb} simplifies, in this context, results of \cite[Chapter $3$]{BruniThesis}.

\subsection{The Lebesgue constant: lower bounds and convergence}
The Lebesgue constant is of additional interest in polynomial approximation as it plays an important role in convergence analysis. For an increasing family $\{\mathcal{S}_r\}_{r \in \N}$ of segments with $|\mathcal{S}_r| = r$, it can be used to estimate the convergence of the interpolant $\Pi_r(\mathcal{S}_r) f$ towards the function $f$. As in the nodal setting, the convergence for all continuous $f \in C(I)$ can in general not be guaranteed. To see this, we note that $\Pi_r(\mathcal{S}_r)$ is a linear projection from $C(I)$ into  $\mathbb{P}_{r-1}$ such that $\Pi_r(\mathcal{S}_r) p (x) = p(x)$ for all $p \in \mathbb{P}_{r-1}$. For general projection operators, and thus also for $\Pi_r(\mathcal{S}_r)$, we know that (see \cite[p. 214]{Cheney66}) 
\begin{equation} \label{eq:atleastloggrowth}
\Lambda_r(\mathcal{S}_r) \geq \Vert \Pi_r(\mathcal{S}_r)\Vert_{\mathrm{op}} > \frac{1}{2} \left( \frac{4}{\pi^2} \ln r - 1 \right).    
\end{equation}
Thus, by contradiction, the Banach-Steinhaus theorem implies that $\lim_{r \to \infty} \Pi_r(\mathcal{S}_r) f = f$ is impossible for all $f \in C(I)$. To guarantee convergence, stronger assumptions on $f$ are necessary. Introducing the modulus of continuity of $f$ as 
$\omega(f,\delta) \doteq \sup_{|x - y| \leq \delta} |f(x) - f(y)|$, we have the following result.

\begin{proposition} \label{prop:convergenceDL}
For given $\mathcal{S} = \{s_1, \ldots, s_r\}$, let $\displaystyle h = \sup_{x \in I} \min_{1 \leq i \leq r} \sup_{y_i \in s_i} |x - y_i|$. Then,
\[\Vert f - \Pi_r(\mathcal{S}) f \Vert \leq \Lambda_r(\mathcal{S}) \, \omega(f,h). \]
\end{proposition}

\begin{proof}
Using the Lagrange basis and the fact that constant functions are reproduced, we get
\[f(x) - \Pi_r(\mathcal{S}) f(x) = \sum_{i=1}^r \left( \int_{s_i} (f(x) - f(y)) \de y \right) \ell_{s_i}(x).\]
By the definition of the modulus $\omega(f,\delta)$ and the value $h$, we therefore get for every $x \in I$:
\[ \left| f(x) - \Pi_r(\mathcal{S}) f(x) \right| \leq \sum_{i=1}^r \omega(f,h) |s_i| |\ell_{s_i}(x)| = \Lambda_r(\mathcal{S}) \omega(f,h).\]
\end{proof}
This result implies that the convergence of $\Pi_r(\mathcal{S}_r) f$ towards $f$ certainly depends on the choice of the segments $\mathcal{S}_r$ (via the Lebesgue constant $\Lambda_r(\mathcal{S}_r)$ and the value $h$), but also on the regularity of $f$ in the sense that $\Lambda_r(\mathcal{S}) \omega(f,h)$ has to converge towards $0$ in order to obtain convergence.

\section{Estimates of the Lebesgue constants} \label{sect:EstLeb}

In this section we aim at estimating the Lebesgue constants $\Lambda_r(\mathcal{S})$ for different families of segments in the classes \ref{itm:one}, \ref{itm:two} and \ref{itm:three}. 

\subsection{The class \ref{itm:one}: estimates for equidistant segments} 

We start with a result that relates the nodal Lebesgue constant for the endpoints $\mathcal{X} = \{\xi_0, \xi_1, \ldots, \xi_r\}$ with the corresponding constant $\Lambda_r(\mathcal{S})$ associated with the $r$ segments $s_i = [\xi_{i-1}, \xi_i]$, $i \in \{1, \ldots r\}$. 

\begin{proposition} \label{prop:upperbound}
	Let $ \Lambda_{r+1} (\mathcal{X}) $ be the Lebesgue constant associated with a set $\mathcal{X}$ of nodes of the form $ a = \xi_0 < \xi_1 < \cdots < \xi_r = b $ and $ \Lambda_r (\mathcal{S})$ the generalized Lebesgue constant associated with the segments $s_i = [\xi_{i-1}, \xi_i]$ constructed over the same nodes. Then, one has
	$$ \Lambda_{r} (\mathcal{S}) \leq \frac{2}{b-a} \left( \max_{1 \leq i \leq r }\!|s_i|\right) \, r^3 \Lambda_{r+1} (\mathcal{X}) .$$
\end{proposition} 

\begin{proof}
Using the definition \eqref{eq:genLeb} of the Lebesgue constant $\Lambda_r (\mathcal{S})$, and denoting $h = \max_{1 \leq i \leq r }\!|s_i|$, we obtain the following first estimate
\begin{equation*} 
\Lambda_r (\mathcal{S}) \leq h \sup_{x \in I} \sum_{j = 1}^r |\ell_{s_j}(x)|. 
\end{equation*}
    Now, plugging in the characterisation \eqref{eq:lagrange1form} of the Lagrange basis in the class \ref{itm:one}, we get
        \begin{align*}
            \Lambda_r (\mathcal{S}) & \leq h \sup_{x \in I} \sum_{j = 1}^r |\ell_{s_j}(x)| = h \sup_{x \in I} \sum_{j = 1}^r \left\vert\sum_{i=j}^{r} \ell_{\xi_i}'(x) \right\vert  \leq h \sup_{x \in I} \sum_{j = 1}^r \sum_{i=j}^{r} \vert \ell_{\xi_i}'(x) \vert \leq h r \sup_{x \in I} \sum_{j = 1}^r \vert \ell_{\xi_j}'(x) \vert.
    \end{align*}
    Since the functions $ \vert \ell_{\xi_i}'(x) \vert $ are piecewise polynomials of degree $r$, we can apply the Markov brothers' inequality (cf. \cite[p. 300]{Achieser92} or \cite[p. 91]{Cheney66}) composed with the affinity $ \Psi: [a,b] \to [-1,1] $ to all terms in the sum on the right hand side and obtain the final estimate
    \begin{equation*}
        \Lambda_r (\mathcal{S}) \leq \frac{2h}{b-a} r^3 \sup_{x \in I} \sum_{j = 0}^r \vert \ell_{\xi_j}(x) \vert = \frac{2h}{b-a} r^3 \Lambda_{r+1} (\mathcal{X}).
    \end{equation*}
    Note that the right hand side may be trivially bounded by $ 2 r^3 \Lambda_{r+1} (\mathcal{X}) $.
\end{proof}

This result can be synthesised as follows: if the Lebesgue constant associated with a collection of segments shows an exponential behaviour, the corresponding nodes are comparably ill-conditioned.

\subsubsection{Uniform segments} A classical results states that the nodal Lebesgue constant $\Lambda_r(\mathcal{X}^{\mathrm{eq}})$ grows exponentially if equidistant points $\mathcal{X}^{\mathrm{eq}}$ are used as interpolation nodes. This can be proved by direct computation \cite{Schoenage} or via the characterisation of the Lebesgue constant as the norm of the interpolation operator \cite{Trefethen}. In the case of segments, i.e. for $ \Lambda_r (\mathcal{S}^{\mathrm{eq}}) $, this has only been observed numerically \cite{BruniEdges, BruniThesis} so far. We prove it now for $ \Lambda_r(\mathcal{S}^{\mathrm{eq}})$ and offer first an exponential lower bound.

\begin{lemma} \label{lem:lowerbound}
    We have
    $$ \Lambda_r (\mathcal{S}^{\mathrm{eq}}) > \frac{1}{\pi} \frac{2^{r-1}}{r^2} .$$
\end{lemma}

\begin{proof} 
It is shown in \cite[Proposition 4.2]{BruniThesis} that in the case \ref{itm:one} of concatenated segments and for a differentiable function $f$ one has the following relation between the interpolation operator $\Pi_r(\mathcal{S})$ in the segmental setting and the interpolation operator $\Pi_{r+1}(\mathcal{X})$ in the nodal setting:
\begin{equation} \label{eq:commutationderivative}
(\Pi_r(\mathcal{S}) f')(x) = (\Pi_{r+1}(\mathcal{X}) f)'(x), \quad x \in I. 
\end{equation} 
Let us now consider the function $ f(x) = e^{\imath \pi x} $ and its derivative $ f'(x) = \imath \pi e^{\imath \pi x} = \imath \pi f(x)$. As $\|f'\| = \pi$, we get
    $$ \Lambda_r(\mathcal{S}^{\mathrm{eq}}) \geq \frac1{\pi} \Vert \Pi_r(\mathcal{S}^{\mathrm{eq}}) f' \Vert = \frac1{\pi} \Vert (\Pi_{r+1}(\mathcal{X}^{\mathrm{eq}}) f )'\Vert \geq \frac{1}{\pi} \frac{|(\Pi_{r+1}(\mathcal{X}^{\mathrm{eq}}) f)(0.5) - 1|}{0.5}   > \frac{1}{\pi} \frac{2^{r-1}}{r^2} ,$$
    where we used \eqref{eq:commutationderivative} in the second step and the inequality $(\Pi_{r+1}(\mathcal{X}^{\mathrm{eq}}) f)(0.5) < - \frac{2^{r-2}}{r^2}$ proven in \cite{Trefethen} in the last step.
\end{proof}

Lemma \ref{lem:lowerbound} shows that the generalised Lebesgue constant associated with uniform segments grows at least exponentially. The following result is an immediate consequence of Proposition \ref{prop:upperbound} and helps in obtaining the opposite bound. 

\begin{lemma} \label{lem:upperbound}
	One has
	$$ \Lambda_r(\mathcal{S}^{\mathrm{eq}}) \leq 2 r^2 \Lambda_{r+1}(\mathcal{X}^{\mathrm{eq}}), $$
    where $\Lambda_{r+1}(\mathcal{X}^{\mathrm{eq}})$ is the nodal Lebesgue constant associated with the  uniform nodes $\mathcal{X}^{\mathrm{eq}}$ on $I$.
\end{lemma} 

\begin{proof}
    By the invariance of the generalised Lebesgue constant proved in Proposition \ref{prop:invarianceLeb}, this is immediate from Proposition \ref{prop:upperbound} with the interval $ I = [-1,1]$, so that $ h = | s_i | = \frac{2}{r} $ for each $ i $.
\end{proof}

Several sharp estimates are known in literature for $ \Lambda_{r+1} (\mathcal{X}^{\mathrm{eq}})$, see \cite{Brutman1,Trefethen} and the references therein. Lemma \ref{lem:upperbound} hence provides an exponential upper bound for $\Lambda_r (\mathcal{S}^{\mathrm{eq}})$, ensuring an exponential growth of the generalised Lebesgue constant associated with uniform segments in $I$. As final estimates, we may collect the results above and offer
\begin{equation} \label{eq:uniformbounds}
    \frac{1}{\pi} \frac{2^{r-1}}{r^2}\leq \Lambda_r (\mathcal{S}^{\mathrm{eq}}) \leq r 2^{r+4}
\end{equation}
as boundaries of  $\Lambda_r (\mathcal{S}^{\mathrm{eq}})$.
These bounds are visualized in Fig. \ref{fig:uniformestimates}, by computing the Lebesgue constants $\Lambda_r (\mathcal{S}^{\mathrm{eq}})$ up to degree $ 40 $. An implementation strategy for this case can be found in \cite{ComputingWeights}.

\begin{figure}[h]
    \centering
    \includegraphics[width = 7.5cm]{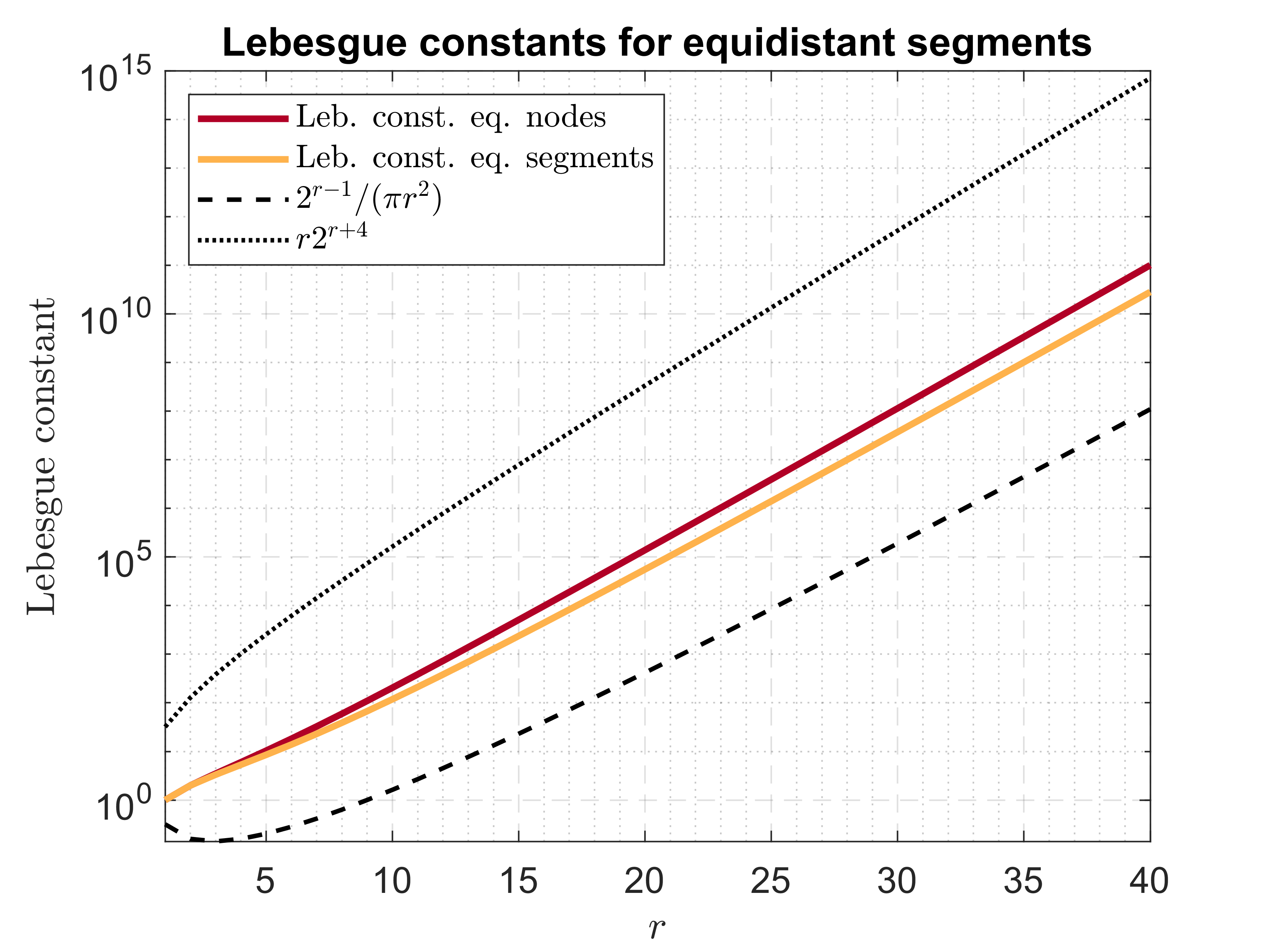} 
    \includegraphics[width = 7.5cm]{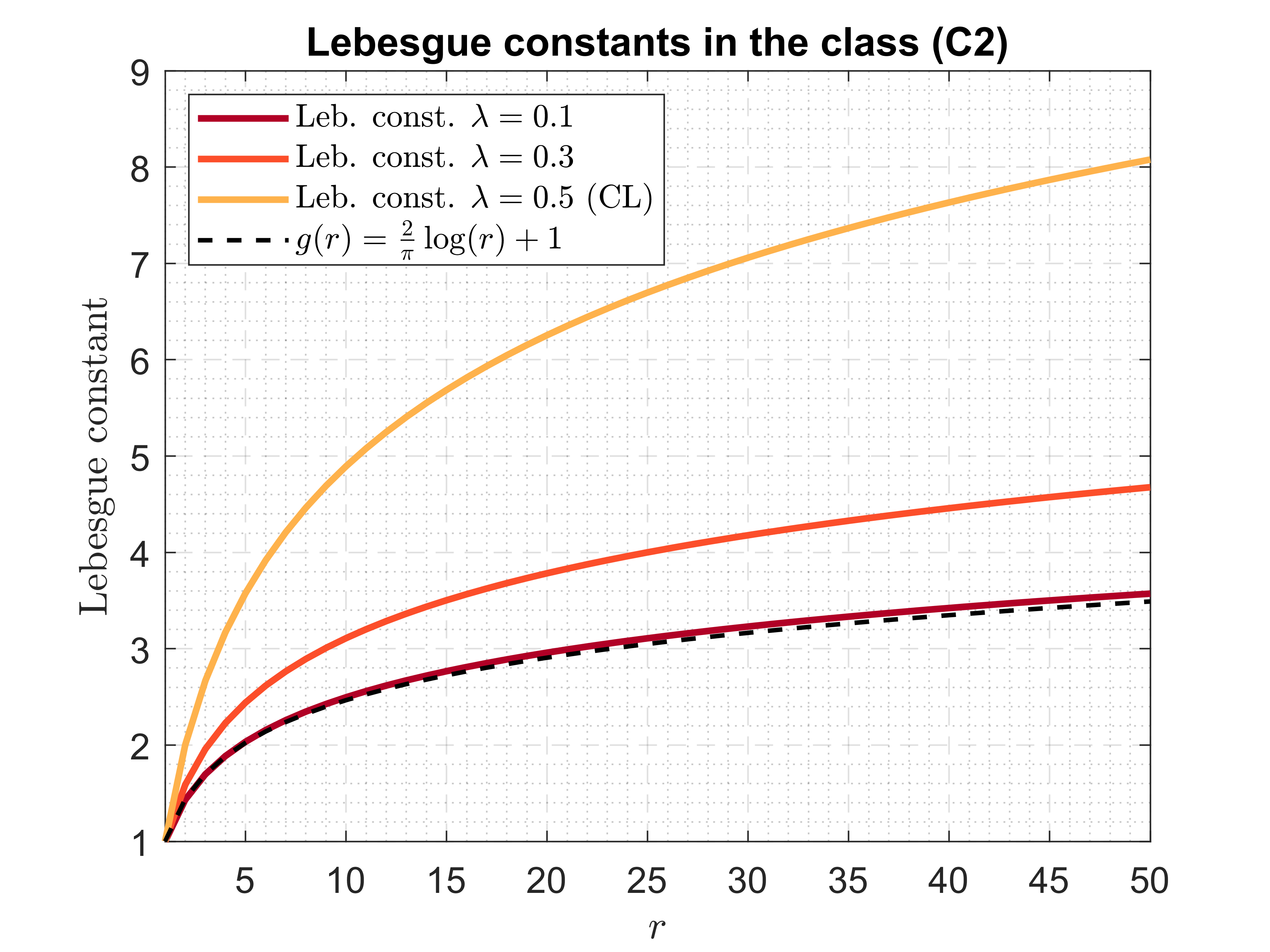}
    \vspace{-.7cm}
    \caption{Left: Lebesgue constants for $r$ equidistant nodes and segments in  \ref{itm:one}. The dashed and dotted lines indicate the derived bounds. Right: Lebesgue constants for $r$ segments in the class \ref{itm:two} for different values of $\lambda$. The CL segments correspond to $\lambda = 0.5$. The dashed line approximates the nodal Lebesgue constant for $\lambda \to 0$.}
    \label{fig:uniformestimates}
\end{figure}

\subsection{The class \ref{itm:two}: condition numbers for segments with constant arc-length}

In the class \ref{itm:two}, the central formula linking the given data $\mu_i$ with the interpolating polynomial $p_{r-1}$ is given by Eq. \eqref{eq:interpcondC2}. We can formulate the generalized interpolation process with help of two linear mappings. In the first map, we consider the point set 
$\mathcal{T} = \{ \cos \tau_i \; : \; i \in \{1, \ldots, r\}\}$ describing the arc-midpoints of the segments $s_i$. We then define the linear operator
\begin{align} \label{eq:operatorphi}
	\Phi_{r}: \quad L_{\infty}(I) &\to \mathbb{P}_{r-1}, \quad 
	f \mapsto \sum_{i=1}^r \frac{1}{|s_i|}\mu(f,s_i) \ell_{\cos \tau_i}(x) = q_{r-1}(x),
\end{align}
This operator provides the unique polynomial $q_{r-1}$ that interpolates the data $\frac{1}{|s_i|}\mu(f,s_i)$ on the arc-midpoints $\cos \tau_i$. The respective uniform norm $\Vert \Phi_{r} f \Vert$ can be bounded from above by
\[ \|\Phi_{r} f \| = \sup_{x \in I} \left\vert \sum_{i=1}^r  \frac{\mu(f,s_i)}{|s_i|}\ell_{\cos \tau_i}(x)\right\vert \leq \max_{1 \leq i \leq r} \frac{|\mu(f,s_i)|}{|s_i|} \sup_{x \in I}  \sum_{i=1}^r  \left\vert\ell_{\cos \tau_i}(x)\right\vert \leq \Lambda_r(\mathcal{T}) \|f\|. \]
In other words, the operator norm $\|\Phi_{r} \|_{\mathrm{op}}$ is bounded by the Lebesgue constant of the (nodal) interpolation operator for the nodes $\mathcal{T}$. As a second linear operator, we consider the integral operator $K_{\tau}: L_{\infty}(I) \to C(I)$ defined for $0 < \rho < \pi$ and $0<t<\pi$ by
\[K_{\rho} f(\cos t) \doteq \frac{1}{\cos (t - \rho) - \cos(t + \rho)} \int_{\cos(t + \rho)}^{\cos (t - \rho)}f(x) \de x.\]
For the endpoints $t \in \{0,\pi\}$, we can define $K_{\rho} f$ by taking the respective limits. For the Chebyshev polynomials $U_j(x)$ we further get
\begin{align} \label{eq:spectraldecompositionK}  K_{\rho} U_j(\cos t) &= \frac{1}{\cos (t - \rho) - \cos(t + \rho)} \int_{\cos(t + \rho)}^{\cos (t - \rho)} U_j(x) \de x \\ &= \frac{1}{j+1}\frac{T_{j+1}(\cos (t - \rho)) - T_{j+1}(\cos (t + \rho))}{\cos (t - \rho) - \cos(t + \rho)} \notag \\ 
& = \frac{1}{j+1}\frac{\sin((j+1)t) \sin((j+1)\rho)}{\sin(t) \sin(\rho)} = \frac{1}{j+1}\frac{\sin((j+1)\rho)}{\sin(\rho)} U_j(\cos t). \notag
\end{align}
This means that the Chebyshev polynomials $U_j$ are eigenfunctions of the integral operator $K_{\rho}$ with respect to the eigenvalues $\frac{1}{j+1}\frac{\sin((j+1)\rho)}{\sin(\rho)}$ and that $K_{\rho}$ maps $\mathbb{P}_{r-1}$ into $\mathbb{P}_{r-1}$ for every $r$. Further, $K_{\rho}$ is invertible on $\mathbb{P}_{r-1}$ if $\rho < \pi/r$. We denote the respective restriction by $K_{\rho,r} \doteq K_{\rho}\vert_{\mathbb{P}_{r-1}}$. In view of the interpolation problem, the operator $K_{\rho,r}$ maps the interpolation polynomial $p_{r-1}$ onto the polynomial $q_{r-1}$. Using the two operators $K_{\rho,r}$ and $\Phi_{r}$, we can now formulate the generalized interpolation problem on the segments $\mathcal{X}$ of the class \ref{itm:two} as 
\[ p_{r-1}(x) = \Pi_r(\mathcal{S}) f (x) = K_{\rho,r}^{-1} \Phi_{r} f (x).\]
This description of the interpolation separates the impact of the arc-midpoints from the arc-radius $\rho$ of the segments. For the numerical conditioning in the class \ref{itm:two} we get now the following result.

\begin{proposition} \label{prop:upperboundC2}
Let $\mathcal{S}$ be a set of segments in the class \ref{itm:two} with constant arc-radius $0 < \tau < \frac{\pi}{r}$. Further, let $ \Lambda_{r} (\mathcal{T}) $ be the Lebesgue constant associated with the set $\mathcal{T}$ of arc-midpoints $\cos \tau_i$ of the segments $s_i = [\cos(\tau_i + \rho), \cos(\tau_i - \rho)]$. Then, the numerical conditioning of the interpolation problem associated to the segments $\mathcal{S}$ is bounded by  
$$ \|\Pi_r(\mathcal{S})\|_{\mathrm{op}} \leq \, \|K_{\rho,r}^{-1}\|_{\mathrm{op}} \, \Lambda_{r}(\mathcal{T}).$$ 
\end{proposition} 

In order to guarantee a low order conditioning of the interpolation in the class \ref{itm:two}, it is therefore important that the nodal Lebesgue constant $\Lambda_{r}(\mathcal{T})$ for the arc-midpoints $\mathcal{T}$ as well as the norms $\|K_{\rho,r}^{-1}\|_{\mathrm{op}}$ are small. The following lemma ensures that if the arc-radius $\rho$ is properly adapted to the number $r$ of segments, then the norm $\|K_{\rho,r}^{-1}\|_{\mathrm{op}}$ is uniformly bounded. 

\begin{lemma} \label{lem:upperboundnormK}
    For $\rho(r) = \frac{\lambda \pi}{r}$ with $0 < \lambda < 1$, the operator norms $\|K_{\rho(r),r}^{-1}\|_{\mathrm{op}}$ are uniformly bounded in $r \in \N$. The respective explicit bound depends only on the parameter $\lambda$.
\end{lemma}

\begin{proof} To show the uniform boundedness, we will make use of the spectral decomposition \eqref{eq:spectraldecompositionK} of  $K_{\rho(r),r}$, as well as a result of Vinogradov for multipliers in Chebyshev expansions. For $0<\lambda<1$, we define the function $\varphi_{\lambda}$ on $[0,\infty)$ as
\[ \varphi_{\lambda}(u) = \left\{ \begin{array}{ll}
     \frac{\lambda u \pi}{\sin(\lambda u \pi)} & \text{if $u \leq 1$},\\ (2 - u) \frac{\lambda \pi}{\sin(\lambda \pi)}
     & \text{if $1 < u \leq 2$} \\
     0 & \text{$u > 2$}.
\end{array} \right.\]
As $\varphi_\lambda$ is continuous and compactly supported on $[0,2]$, we have $\varphi_\lambda \in C([0,\infty))$, $ \| u \varphi_{\lambda}(u) \|_2 < \infty $ and $\sum_{j=1}^\infty \textstyle\left( \frac{j}{r} \varphi_{\lambda}(\frac{j}{r})\right)^2 < \infty$ for all $r \in \mathbb{N}$. For a continuous function $f$ with the expansion $f = \sum_{j=0}^\infty b_j U_j$ we define the operator $\mathcal{M}_{\lambda,r}: C(I) \to C(I)$ by 
\[\mathcal{M}_{\lambda,r} f(x) = \sum_{j=0}^\infty \varphi_{\lambda}(\textstyle \frac{j+1}{r})b_j U_j(x).\]
Then, by \cite[Theorem 5]{Vinogradov01}, the operators $\mathcal{M}_{\lambda,r}$ are uniformly bounded by
\begin{equation} \label{eq:explicitvaluenormM}
  \sup_{r \in \N }\|\mathcal{M}_{\lambda,r}\|_{\mathrm{op}} = \lim_{r \to \infty }\|\mathcal{M}_{\lambda,r}\|_{\mathrm{op}} =\frac{2}{\pi} \int_{0}^\infty z \left\vert \int_0^\infty u \varphi_{\lambda}(u) \sin (z u) \de u \right\vert \de z,  
\end{equation} 
the last integral being finite with the stated properties of the function $\varphi_{\lambda}$. Now, note that by the spectral decomposition \eqref{eq:spectraldecompositionK} we have for $p \in \mathbb{P}_{r-1}$ and $\rho(r) = \frac{\lambda \pi}{r}$ the identity 
\[ \mathcal{M}_{\lambda,r} p(x) = \frac{r\sin\frac{\lambda \pi}{r} }{ \lambda \pi} (K_{\rho(r),r}^{-1} p) (x).\]
We therefore get for the operator norms
\begin{equation} \label{eq:upperboundoperatorK}
    \sup_{r \in \N} \|K_{\rho(r),r}^{-1} \|_{\mathrm{op}} \leq \frac{\lambda \pi}{\sin \lambda \pi} \: \sup_{r \in \N}\|\mathcal{M}_{\lambda,r}\|_{\mathrm{op}}
\end{equation}
and, thus, the statement of the lemma.
\end{proof}

\begin{remark} Proposition  \ref{prop:upperboundC2} and Lemma \ref{lem:upperboundnormK} imply that if the arc-radius $\rho$ is decreasing proportionally to $1/r$ then the asymptotic growth of the numerical condition number $\|\Pi_r(\mathcal{S})\|_{\mathrm{op}}$ depends in essence on the growth of the nodal Lebesgue constant on the arc-midpoints $\mathcal{T}$. 

The spectral radius of the operator $K_{\rho(r),r}^{-1}$, $\rho(r) = \frac{\lambda \pi}{r}$ is given as $\frac{\lambda \pi}{\sin \lambda \pi}$. Combined with \eqref{eq:upperboundoperatorK}, we can therefore estimate $\|K_{\rho(r),r}^{-1} \|_{\mathrm{op}}$ as
\begin{equation} \label{eq:upperlowerboundoperatorK}
    \frac{\lambda \pi}{\sin \lambda \pi} \leq \sup_{r \in \N} \|K_{\rho(r),r}^{-1} \|_{\mathrm{op}} \leq \frac{\lambda \pi}{\sin \lambda \pi} \: \sup_{r \in \N}\|\mathcal{M}_{\lambda,r}\|_{\mathrm{op}}.
\end{equation}
Thus, as soon as $\lambda$ approaches $1$, the norms $\|K_{\rho(r),r}^{-1}\|_{\mathrm{op}}$ increase rapidly. For $\rho(1) = \frac{\pi}{r}$ the operator $K_{\pi/r,r}$ is not invertible and the respective segment set $\mathcal{S}$ not unisolvent in $\mathbb{P}_{r-1}$. This loss of uniqueness was already predicted in Proposition \ref{prop:uniquenessC2}. In the limit $\lambda \to 0$, we get $K_{0,r} = \mathrm{id}$ the identity on $\mathbb{P}_{r-1}$. This corresponds to the case of nodal interpolation on the set $\mathcal{T}$.  
\end{remark}

\begin{corollary} \label{cor:logarithmicgrowth}
Let $\mathcal{S}^{\mathrm{CL}}$ be the Chebyshev-Lobatto segments in $I$. Then, the Lebsgue constant $\Lambda_r(\mathcal{S}^{\mathrm{CL}}) = \| \Pi_r(\mathcal{S}^{\mathrm{CL}})\|_{\mathrm{op}}$ is bounded by
\[ \frac1{2} \left( \frac{4}{\pi^2} \ln r - 1 \right) \leq 
\Lambda_r(\mathcal{S}^{\mathrm{CL}}) \leq  \left( \ln r + \frac{\pi}{2} \right)\, \sup_{r \in \N}\|\mathcal{M}_{1/2,r}\|_{\mathrm{op}} , \]
with a formula for the supremum $\sup_{r \in \N}\|\mathcal{M}_{1/2,r}\|_{\mathrm{op}}$ given in \eqref{eq:explicitvaluenormM}.
\end{corollary}
A similar logarithmic growth of the Lebesgue constant is also obtained if we use segments with the same arc-midpoints as in $\mathcal{S}^{\mathrm{CL}}$ but with a different arc-radius $\rho = \lambda \frac{\pi}{r}$, $\lambda< 1$. The Lebesgue constants for different values of the parameter $\lambda$ are compared in Fig. \ref{fig:uniformestimates}.  

\begin{proof}
For the CL segments, the arc-midpoints $\mathcal{T}$ correspond to the standard Chebyshev nodes on $I$. A close upper bound for the corresponding nodal Lebesgue constant is given by $\Lambda_r(\mathcal{T}) \leq \frac{2}{\pi} \ln r + 1$ (cf. \cite{Brutman1}). This, together with Proposition  \ref{prop:upperboundC2} and Eq. \eqref{eq:upperboundoperatorK}, yields the statement. 
\end{proof}

From Proposition \ref{prop:convergenceDL} and Corollary \ref{cor:logarithmicgrowth} we get finally the following convergence result.

\begin{corollary}
 Let $\mathcal{S}^{\mathrm{CL}}$ be the Chebyshev-Lobatto segments in $I$ and $f \in C(I)$. We have
$$ \|\Pi_r(\mathcal{S}^{\mathrm{CL}}) f - f\|_\infty \leq \Lambda_r(\mathcal{S}^{\mathrm{CL}}) \, \omega(\pi/r,f).$$
In particular, if $f$ satisfies a Dini-Lipschitz condition of the form $\omega(\frac1r,f) \ln r \to 0$ as $r \to \infty$, then $\Pi_r(\mathcal{S}^{\mathrm{CL}}) f$ converges uniformly to $f$.    
\end{corollary}

\begin{figure}[h]
    \centering
    \includegraphics[width = 7.5cm]{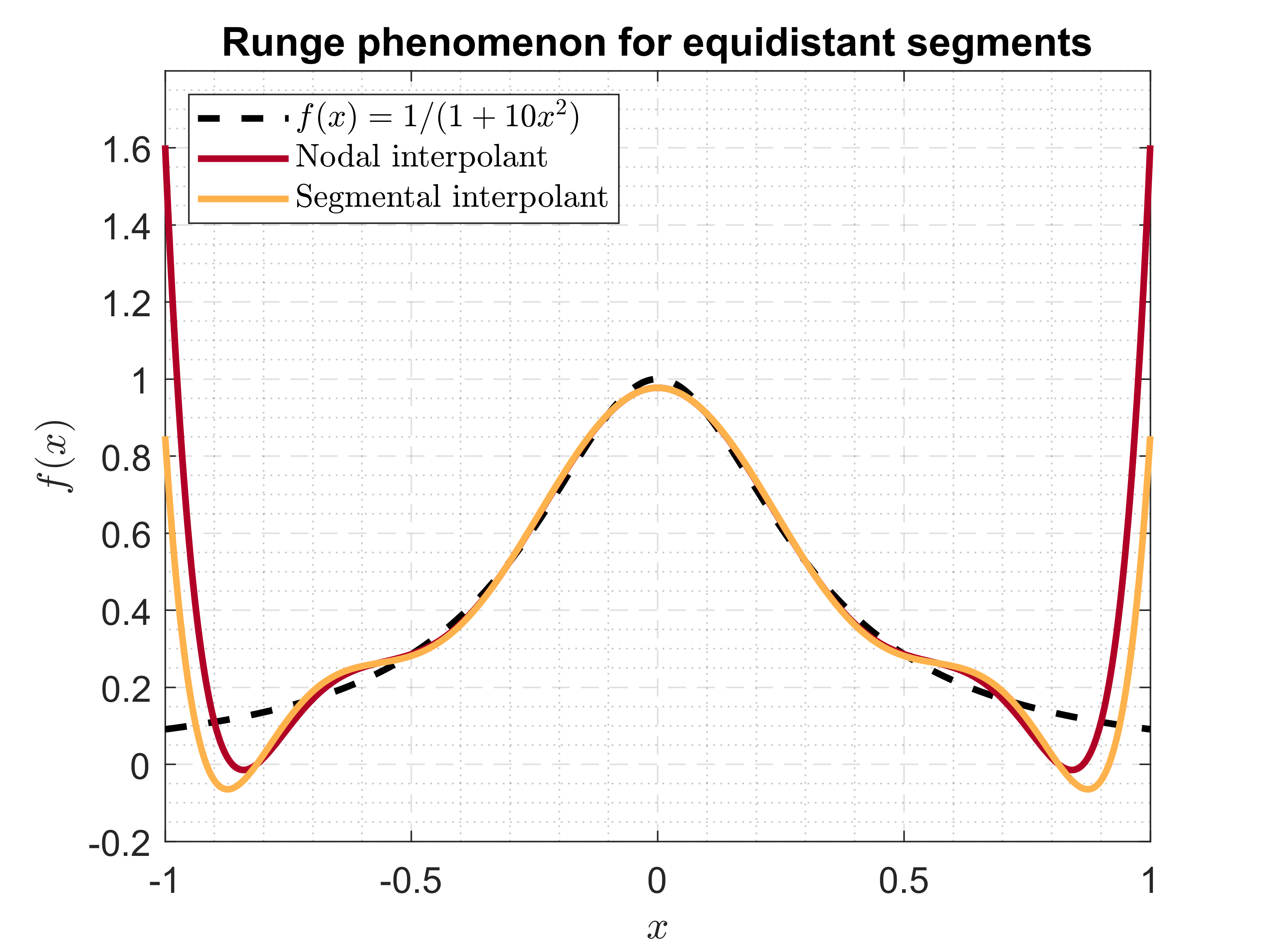} 
    \includegraphics[width = 7.5cm]{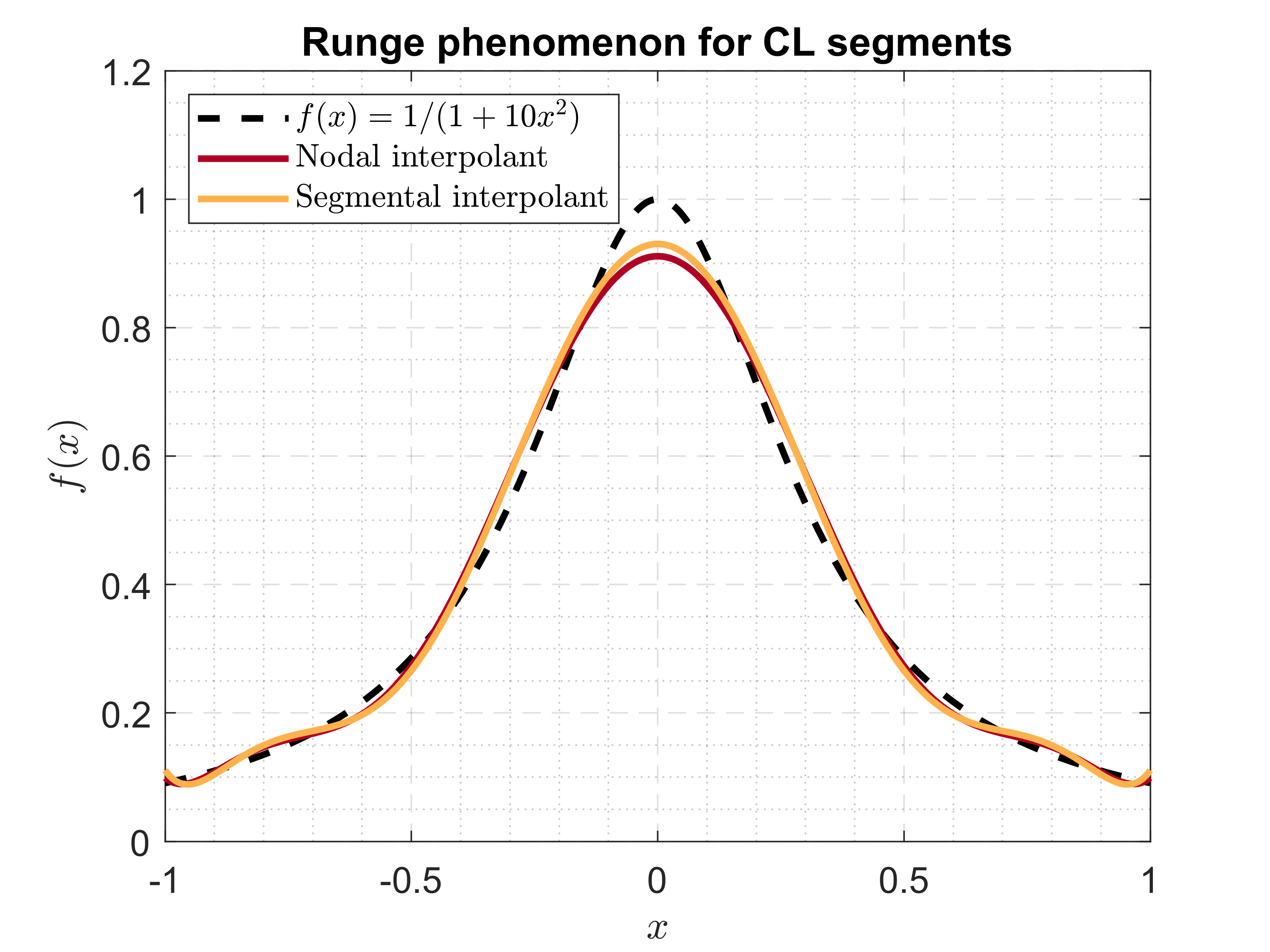}
    \vspace{-.5cm}
    \caption{Effects of bad conditioning in the example of the Runge function $f(x) = 1/(1+10x^2)$. Left: interpolation of the Runge function using $10$ equidistant segments (and $10$ equidistant nodes). Right: polynomial interpolation of $f$ with $10$ CL segments (and $10$ Chebyshev nodes).}
    \label{fig:rungephenomenon}
\end{figure}

\section{A Whitney forms perspective} \label{sect:WFperspective}

The interest in the generalised Lebesgue constant \eqref{eq:genLeb} is motivated by \emph{Whitney forms} \cite{Whitneybook}. This class of differential forms is widely used in finite element approximation \cite{AFW06, BossavitBook} and, more recently, has been also considered for pure interpolation purposes \cite{BruniRunge}. Whitney forms are a relevant space of \emph{polynomial differential forms}; in particular, they naturally come at play when one considers \emph{simplicial elements}.

As integration on $k$-simplices is a natural operation on differential $k$-forms, the interpolation problem onto the space of Whitney forms $ \mathcal{P}_r^- \Lambda^k $ is sharply understood in terms of \emph{weights} \cite{RB09}: for any (sufficiently regular) form $ \omega \in \Lambda^k $, find $ \Pi \omega \in \mathcal{P}_r^- \Lambda^k $ such that
$$ \int_{s_i} \omega = \int_{s_i} \Pi \omega \qquad \forall s_i \in S_r^k .$$
Here $ S_r^k $ is thus a collection of $k$-simplices, usually referred to as \emph{small simplices} \cite{Rapetti07}, and the selection of this set seriously affects the reliability of the interpolator $ \Pi $ \cite{BruniRunge, BruniThesis} (an example is here depicted in Fig. \ref{fig:rungephenomenon}). A measure of the quality of this operator is in fact the \emph{generalised Lebesgue constant} \cite{GenLeb}, whose definition is stated in terms of differential forms as
\begin{equation} \label{eq:generalisedLebconst}
	\Lambda_r^k \doteq \sup_{c \in \mathcal{C}_r^k} \frac{1}{|c|} \sum_{s_i \in S_r^k} | s_i | \left\vert \int_c \omega_{s_i} \right\vert .
\end{equation}
In this equation, the supremum is sought over the collection of $ k$-chains $ \mathcal{C}_r^k $ and $ \omega_{s_i} $ is the basis in the duality induced by weights over the collection of support $ S_r^k $. 
The characterisation of the set $ \mathcal{C}_r^k $ significantly rises the complexity of the problem and creates a relevant obstacle to the proof of theoretical bounds. Generalised Lebesgue constants have in fact been only numerically estimated; some computations can be found in \cite{BruniEdges,BruniFaces,BruniThesis}.

When $ k = n = 1 $, there is an identification of $ 1 $-forms with functions that associates any $ 1$-form $ f(x) \de x $ simply with $ f(x) $. Plugging this into Eq. \eqref{eq:lagrange1form}, it is hence immediate to deduce that
$$ \omega_{s_j} = \sum_{k=j}^r \ell_{\xi_k}'(x) \de x  $$
is the $j$-th element of the cardinal basis of $ \mathcal{P}_r^- \Lambda^1 $. This, together with the characterisation of the generalised Lebesgue constant in terms of averages on segments given in Section \ref{sect:Lebconstinterpsegment}, proves that Eq. \eqref{eq:generalisedLebconst} boils down to Eq. \eqref{eq:genLeb}, overcoming the impasse of the estimation of the factor $ | c | $ appearing at the denominator. As a consequence, one deduces that all the results stated and proved in this work for segments and averages can be identically cast in the language of differential forms. Thanks to the equality $ \mathcal{P}_r^- \Lambda^1 = \mathcal{P}_{r-1} \Lambda^1 $ (valid on the real line) one furthermore deduces that these results not only apply to Whitney forms but exhaust the whole case of $ 1 $-forms.

\section{Conclusions} \label{sec-conclusion}

In this work, we studied a generalisation of polynomial interpolation in which the fitted data consists of function averages over interval segments. We provided conditions for unisolvence and explicit formulae for Lagrange-type bases. Furthermore, we derived theoretical bounds for the growth of the Lebesgue constant linked to segmental interpolation. 
The behavior of these bounds resembles well-known results from classical nodal interpolation, although definitions and techniques sensibly diverge. In this work, we confined ourselves to the one dimensional setting, and we expect to generalise these results, using the language of differential forms, in forthcoming works. The developed techniques will be further useful to transfer related interpolation and quadrature methods as for instance mapped basis approaches \cite{platte,DeMarchi2020} or polynomial quadrature rules \cite{Cappellazzo2023} from a nodal to a segmental setting.

\section*{Acknowledgements}

This research has been accomplished within the research networks RITA and UMI-TAA, and was partially funded by GNCS-IN$\delta$AM. The first author is funded by IN$\delta$AM and supported by Universit\`a di Padova.

\end{document}